\documentclass[preprint,EJP]{ejpecp}
\pdfoutput=1
\hypersetup{colorlinks,citecolor=blue,urlcolor=blue,unicode,verbose}
\usepackage[inline]{enumitem}   
\usepackage{array}							
\usepackage{mathrsfs}						
\usepackage{stmaryrd}					  
\usepackage{yhmath}							
\usepackage{accents}						
\usepackage{xfrac}              
\usepackage{epigraph}

\SHORTTITLE{Simplified stochastic calculus via semimartingale representations}
\TITLE{\vspace{-5cm}Simplified stochastic calculus via semimartingale representations} 

\AUTHORS{%
  Ale\v{s}~\v{C}ern\'{y}\footnote{Bayes Business School, City, University of London.
    \EMAIL{ales.cerny.1@city.ac.uk}}
  \and 
  Johannes~Ruf\footnote{Department of Mathematics, London School of Economics and Political Science. \EMAIL{j.ruf@lse.ac.uk}}
	}


\AMSSUBJ{60G07; 60G44; 60G48; 60H05; 60H05}

\ABSTRACT{We develop a stochastic calculus that makes it easy to capture a variety of predictable transformations of semimartingales such as changes of variables, stochastic integrals, and their compositions. The framework offers a unified treatment of real-valued and complex-valued semimartingales. The proposed calculus is a blueprint for the derivation of new relationships among stochastic processes with specific examples provided below.}

\DeclareFontFamily{U}{mathx}{\hyphenchar\font45}
\DeclareFontShape{U}{mathx}{m}{n}{
      <5> <6> <7> <8> <9> <10>
      <10.95> <12> <14.4> <17.28> <20.74> <24.88>
      mathx10
      }{}
\DeclareSymbolFont{mathx}{U}{mathx}{m}{n}
\DeclareFontSubstitution{U}{mathx}{m}{n}
\DeclareMathAccent{\widecheck}{0}{mathx}{"71}
\DeclareMathAccent{\wideparen}{0}{mathx}{"75}

\makeatletter
\newcommand*\rel@kern[1]{\kern#1\dimexpr\macc@kerna}
\newcommand*\widebar[1]{%
  \begingroup
  \def\mathaccent##1##2{%
    \rel@kern{0.8}%
    \overline{\rel@kern{-0.8}\macc@nucleus\rel@kern{0.2}}%
    \rel@kern{-0.2}%
  }%
  \macc@depth\@ne
  \let\math@bgroup\@empty \let\math@egroup\macc@set@skewchar
  \mathsurround\z@ \frozen@everymath{\mathgroup\macc@group\relax}%
  \macc@set@skewchar\relax
  \let\mathaccentV\macc@nested@a
  \macc@nested@a\relax111{#1}%
  \endgroup
}
\makeatother

\newcommand{\N}{\mathbb{N}}
\newcommand{\Z}{\mathbb{Z}}

\newcommand{\R}{\mathbb{R}}
\newcommand{\Cx}{\mathbb{C}}

\newcommand{\Cinf}{\widebar{\Cx}}
\renewcommand{\Re}{\operatorname{Re}}
\renewcommand{\Im}{\operatorname{Im}}
\newcommand{\vecr}{\operatorname{vec}^{\rm r}}
\newcommand{\vecc}{\operatorname{vec}^{\rm c}}
\newcommand{\id}{{\operatorname{id}}}

\newcommand{\I}{\mathfrak{I}}
\newcommand{\Uni}{\mathfrak{U}}
\renewcommand{\U}{\mathcal{U}}
\renewcommand{\H}{\mathcal{H}}

\newcommand{\Dom}{\textsf{\upshape Dom}}
\renewcommand{\d}{\mathrm{d}}
\newcommand{\ddp}{\mathrm{dp}}
\newcommand{\q}{\mathrm{qc}}

\DeclareMathOperator{\sgn}{sgn}
\newcommand{\e}{\mathrm{e}}
\newcommand{\tr}{\operatorname{trace}}
\newcommand{\NaN}{\mathrm{NaN}}

\newcommand{\E}{\textsf{\upshape E}}

\renewcommand{\P}{\textsf{\upshape P}}
\newcommand{\Qu}{\textsf{\upshape Q}}
\newcommand{\indicator}[1]{\mathbf{1}_{#1}}

\newcommand{\filt}[1]{\mathfrak{#1}}
\newcommand{\sigalg}[1]{\mathscr{#1}}
\newcommand{\Exp}{\ensuremath{\mathscr{E}}}
\newcommand{\Log}{\ensuremath{\mathcal{L}}}
\newcommand{\V}{\ensuremath{\mathscr{V}}}

\newcommand{\lc}{[\![}

\newcommand*{\bigs}[1]{\scalebox{1.2}{\ensuremath#1}}
\DeclarePairedDelimiter\abs{{\mathopen|}}{{\mathclose|}}

\begin{document}

\section{Introduction}\label{sect: intro}
\setlength\epigraphwidth{.8\textwidth}
\setlength\epigraphrule{0pt}
\epigraph{\emph{``Because in mathematics we pile inferences upon inferences, it is a good thing when\-ever we can subsume as many of them as possible under one symbol. For once we have understood the true significance of an operation, just the sen\-si\-ble appre\-hension of its symbol will suffice to obviate the whole reasoning process that earlier we had to engage anew each time the operation was encountered.''}}{--- Carl Jacobi (1804--1851) \cite[p.~67]{remmert.91}}\medskip

We study the following concept. A semimartingale $Y$ is said to be represented by a semimartingale $X$ if, roughly speaking, there is a predictable function $\xi$ acting on the increments of $X$ such that the increments of $Y$ satisfy $\d Y_t = \xi_t (\d X_t)$, where $\xi_t (\d X_t)$ is given some ``natural'' meaning. Such representation of $Y$ in terms of $X$, if it exists, is measure-invariant. One hopes that common operations on $Y$ yield processes that are again $X$--representable, for example,\smallskip \newline
(i) a stochastic integral $\zeta_t\d Y_t$ ``ought to'' yield
\begin{equation}\label{intro1}
 \zeta_t\d Y_t = \zeta_t \xi_t (\d X_t);
\end{equation}
(ii) for a change of variables by means of some smooth function $f$ it should be true that 
\begin{equation}\label{intro2}
\d f(Y_t) = f(Y_{t-}+\xi_t (\d X_t))-f(Y_{t-}); 
\end{equation}
(iii) for a new process $Z$ such that $\d Z_t = \psi_t (\d Y_t)$ one would like to obtain the composition rule
\begin{equation}\label{intro3}
\d Z_t = \psi_t( \xi_t (\d X_t) ).
\end{equation}
In integral form we shall write, e.g., 
$$Y_t = Y_0+\xi\circ X_t = Y_0 + \int_0^t \xi_s(\d X_s).$$

The purpose of the calculus \eqref{intro1}--\eqref{intro3} is to reduce the computational burden in a generic modelling situation where one starts from a (multivariate) process $X$ whose predictable $\P$--cha\-ra\-cte\-ri\-stics relative to some truncation function 
are given as the primitive input to the problem. The process $X$, which is trivially representable with respect to itself, is transformed by several applications of Propositions~\ref{P:integral} and \ref{P:Ito} and Theorem~\ref{T:composition}, i.e., by the rules \eqref{intro1}--\eqref{intro3} above, to another process $Y$ which is also $X$--representable. In many situations the required end product is the $\P$--drift of $Y$, here denoted $B^Y$; 
e.g., when computing  exponential compensators; see Duffie, Filipovi\'c, and Schachermayer~\cite[Proposition 11.2]{duffie.al.03}. A generic formula for the drift of a represented process, $B^{\xi\circ X}$, is given in Proposition~\ref{P:170822.1}.

We begin with a brief description of some historical background of this paper. The first seeds of measure-invariant stochastic calculus were planted by McKean~\cite{mckean.69}, who would write the classical It\^o formula in the form
\begin{equation}\label{eq:200612.2}
\d f(X_t) = f'(X_t)\d X_t +\frac{1}{2}f''(X_t)(\d X_t)^2,
\end{equation}
and only afterwards substitute the canonical decomposition of $X$ in the first term and the quadratic variation of $X$ in the second term on the right-hand side of~\eqref{eq:200612.2}. With the development of general semimartingale integration it soon became clear that \eqref{eq:200612.2} is fully rigorous, as written, for any continuous semimartingale $X$ and any sufficiently smooth function $f$ on an appropriate domain; see Dol\'eans-Dade and Meyer~\cite[Th\'eor\`eme~8]{doleans-dade.meyer.70}.

From here it is not a big conceptual leap to want to study general transformations of the increments $\d X_t$ by means of some predictable function $\xi$. Precisely this was suggested by \'Emery~\cite{emery.78} together with the notation $\d Y_t = \xi(\d X_t)$ and a specific measure-invariant formula for $\xi(\d X_t)$, for time-constant deterministic $\xi$ and matrix(!)-valued $X$ (because \'Emery's goal at the time was to study the natural exponential of a matrix with stochastic coefficients). As far as we know, nobody has up until now attempted to build a coherent calculus based on \'Emery's formula. See also Remark~\ref{R:predictable variation} for other connections to the literature.

Next, let us offer a flavour of the simplifications the calculus can achieve in conjunction with drift calculations. For example, for real-valued $X$ and $\alpha\in\Cx$ the calculus permits one to write 
\begin{equation}\label{eq:210404.1}
\e^{\alpha (X-X_0)} = \Exp((\e^{\alpha\,\id}-1)\circ X), 
\end{equation}
where $\Exp$ is the Dol\'eans-Dade  stochastic exponential (see~\cite{doleans-dade.70}).
If $X$ has independent increments and the expectation is finite, one then obtains (see \v{C}ern\'{y} and Ruf~\cite[Theorem~4.1]{crIII})
\begin{equation}\label{eq:210404.2}
\E\bigs[\e^{\alpha (X_t-X_0)}\bigs] = \Exp\bigs(B^{(\e^{\alpha\id}-1)\circ X}\bigs)_t.
\end{equation}
When $X$ is a L\'evy process and $\alpha$ is purely imaginary, the right-hand side of \eqref{eq:210404.2} is just the L\'evy-Khintchin formula; see \cite[Corollary~4.2]{crIII}. Useful formulae akin to \eqref{eq:210404.1} are collected in Table~\ref{tab:S4}. 

\begin{table}[t]
	\centering
	\caption{Useful identities involving the power function, stochastic exponential $\Exp$, and stochastic logarithm $\Log$.}\medskip
	{\footnotesize
		\begin{tabular}{lll}
		\hline\\[-2ex]
			Name             & Assumptions                           & Conclusions                                          \\[1ex]
			P\ref{P:190701b}  & \hspace{-0.8em}
			$\begin{array}{l}
			\alpha\in \Z_+; \text{or}\\
			\alpha \in \Z;\ \Delta X \neq -1; \text{or}\\
			\alpha\in \Cx\setminus\Z;\ \Re \Exp(X)>0
			\end{array}$\medskip
			&   $\Exp(X)^\alpha	= \Exp \left( \left((1+\id)^\alpha  -1\right)\circ X \right)$ \\
			P\ref{P:190701b}     & analogous                             &   $\Exp \bigs(X^{(1)}\bigs)^\alpha\Exp \bigs(X^{(2)}\bigs)^\beta 
			= \Exp \left( \left((1+\id_1)^\alpha (1+\id_2)^\beta -1\right)\circ X \right)$ \\[1ex]
      P\ref{P:190701b}      &	\hspace{-0.8em}		
			$\begin{array}{l}
			\alpha,\beta \in \Z;\ X,X_- \neq 0; \text{or}\\
			\alpha,\beta\in \Cx\setminus\Z;\ \ldots
			\end{array}$\medskip                              & $\begin{array}{l}\Log \left(\bigs(X^{(1)}\bigs)^\alpha \bigs(X^{(2)}\bigs)^				  																									\beta\right)\\ \quad =   \left((1+\id_1)^\alpha (1+\id_2)^\beta -1\right)\circ 		 																					\bigs(\Log\bigs(X^{(1)}\bigs),\Log\bigs(X^{(2)}\bigs)\bigs)\end{array}$\\[2ex]
			P\ref{P:190701}      & none                             & 
			(i) $\Log \bigs(\e^{X}\bigs)=(\e^{\id}-1)\circ X$;\quad 
			(ii) $\lvert\Exp (X)\rvert= \Exp\left(\left(\lvert 1+\id\rvert - 1\right) \circ X\right)$\\[1ex]
			P\ref{P:190701}      &  $\Delta X\neq -1$           & 
			(i) $\Exp (X) = \e^{\log (1+\id) \circ X}$;\quad\,\,
			(ii) $\log\lvert\Exp (X)\rvert= \log\lvert 1+\id\rvert \circ X$\\[1ex]
			P\ref{P:190701}     & $\Re\Exp(X)>0$                     &   $\log\Exp (X) = \log (1+\id) \circ X$ \\[1ex]
			P\ref{P:200604}      & \hspace{-0.8em}
			$\begin{array}{l}
			\alpha\in (0,\infty); \text{or}\\
			\alpha\in \Cx\setminus\R_+;\ \Delta X\neq -1
			\end{array}$\medskip
			&   $|\Exp(X)|^\alpha	= \Exp \left( \left(|1+\id|^\alpha  -1\right)\circ X \right)$ \\
			P\ref{P:200604}     & analogous                         &   
			$\mathopen{\bigs|}\Exp \bigs(X^{(1)}\bigs)\mathclose{\bigs|}^\alpha\mathopen{\bigs|}\Exp \bigs(X^{(2)}\bigs)\mathclose{\bigs|}^\beta 
			= \Exp \left( \left(\abs{1+\id_1}^\alpha \abs{1+\id_2}^\beta -1\right)\circ X \right)$ \\[1.5ex]
			P\ref{P:200604}      &$X,X_-\neq 0$                     & $\begin{array}{l}\Log \left(\bigs|X^{(1)}\bigs|^\alpha \bigs|X^{(2)}\bigs|^   																															\beta\right)\\ \quad   =   \left(|1+\id_1|^\alpha |1+\id_2|^\beta -1\right)																									\circ \bigs(\Log\bigs(X^{(1)}\bigs),\Log\bigs(X^{(2)}\bigs)\bigs)\end{array}$\medskip\\[1ex]
			\hline\\[-1ex]
		\end{tabular}
		}
	\label{tab:S4}\vspace*{-2ex}
\end{table}

The calculus can do more. Staying with $X$ that has independent increments, suppose that instead of the natural exponential $\e^{X-X_0}$ in \eqref{eq:210404.1}, the starting object is the stochastic exponential $\Exp(X)>0$. The calculus now provides a formula for the Mellin transform
\begin{equation*}
\E\bigs[\Exp(X)_t^\alpha\bigs] = \E\bigs[\Exp(((1+\id)^\alpha - 1)\circ X)_t\bigs] = \Exp\bigs(B^{((1+\id)^\alpha - 1)\circ X}\bigs)_t.
\end{equation*}
When $\Exp(X)$ is signed, one can evaluate $|\Exp(X)|^\alpha$ and $\sgn(\Exp(X))|\Exp(X)|^\alpha$ separately to obtain
\begin{align}
\E\bigs[|\Exp(X)_t|^\alpha\bigs] ={}& \Exp\bigs(B^{(|1+\id|^\alpha - 1)\circ X}\bigs)_t;\label{eq:210404.4}\\
\E\bigs[\sgn(\Exp(X))|\Exp(X)_t|^\alpha\bigs] ={}& \Exp\bigs(B^{(\sgn(1+\id)|1+\id|^\alpha - 1)\circ X}\bigs)_t;\label{eq:210404.5}
\end{align}
see \cite[Examples~4.4 and 4.5]{crIII}. 
We refer the reader also to the introductory paper \v{C}ern\'{y} and Ruf~\cite{crI}, 
where other concrete illustrations of the calculus are given.%
\footnote{This paper is conceptually different from \cite{crI} in two important respects. First, we provide a unified treatment of real-valued and complex-valued representations where \cite{crI} only considers two ad-hoc non-interacting subsets of representing functions that must be applied separately to real-valued and complex-valued processes, respectively. For example, \cite{crI} cannot handle the formulae \eqref{eq:210404.4} and \eqref{eq:210404.5}. Second, \cite{crI} operates strictly inside a special class $\Uni$, introduced here.}

On the theoretical side, the paper introduces the class $\Uni$ of universal representing functions that are well-behaved with respect to operations \eqref{intro1}--\eqref{intro3}; if one uses only locally bounded integration, change of variables, and composition, one is guaranteed never to leave $\Uni$, which makes the calculus completely straightforward. For example, the representing functions in \eqref{eq:210404.1}--\eqref{eq:210404.5} and also those in Table~\ref{tab:S4} are all in $\Uni$. The most important results pertaining to the class $\Uni$ are highlighted in Table~\ref{tab:I0}.

\begin{table}[t]
	\centering
	\caption{Summary of statements for $\Uni$ --- the class of universal representing functions. Here $\I(X)$ denotes the class of predictable functions for which $\xi\circ X$ is well-defined.}\medskip
	\label{tab:I0}
	{\footnotesize
		\begin{tabular}{lll}
		\hline\\[-2ex]
			Name             & Assumptions         & Conclusions\\[1.4ex]
			P\ref{P: universal}\quad\  & $\xi\in\Uni$;\ $\xi(\Delta X)$ finite & $\xi\in\I(X)$                                          \\[1ex]
			P\ref{P:1}           & none                                     & \hspace{-0.8em}
			$\begin{array}{l}\id,\id^2\in\Uni;\\[0.1ex]
				X=X_0+\id\circ X;\\[0.1ex]
				[X,X]=\id^2\circ X\medskip
			\end{array}$\\[1ex]
			P\ref{P:integral}  & $\zeta$ locally bounded predictable\quad \  & \hspace{-0.8em}
			$\begin{array}{l}
				\zeta\id\in\Uni;\\[0.1ex]
				\zeta\cdot X = \zeta\id \circ X\medskip
			\end{array}$\\
			P\ref{P:Ito}         & $f:\R\to\R$ smooth                    & \hspace{-0.8em}
			$\begin{array}{l}
				f(X_-+\id)-f(X_-)\in\Uni;\\[0.1ex]
				f(X)=f(X_0)+(f(X_-+\id)-f(X_-))\circ X\medskip
			\end{array}$\\
			T\ref{T:composition0}& $\xi,\psi\in\Uni$;\ $\psi(\xi(\Delta X))$ finite  & \hspace{-0.8em}
			$\begin{array}{l}
				\psi(\xi)\in\Uni;\\[0.1ex]
				\xi,\psi(\xi)\in\I(X);\ \psi\in\I(\xi\circ X);\\[0.1ex]
				\psi\circ(\xi\circ X)=\psi(\xi)\circ X\medskip
			\end{array}$\\[1ex]
			\hline\\
		\end{tabular}
		}\vspace{-2ex}
\end{table}

Furthermore, we develop a coherent theory for a wider class $\I(X)$ of representing functions specific to $X$, in which $\Uni$ appears as a special case. Here the ``natural'' composition rules \eqref{intro1} and \eqref{intro3} sometimes fail. We study sufficient conditions for their validity and offer counterexamples when such conditions are not met. The proposed framework does deliver closedness under composition for general stochastic integrals without further assumptions; this and other important properties of the class $\I(X)$ are collected in Table~\ref{tab:I(X)}. 

\begin{table}[t]
	\centering
	\caption{Summary of statements for $\I(X)$ --- the class of predictable functions for which $\xi\circ X$ is well-defined.}\medskip
	\label{tab:I(X)}
	{\footnotesize
		\begin{tabular}{lll}
		\hline\\[-2ex]
			Name\qquad\             & Assumptions         & Conclusions\\[1.4ex]
			P\ref{P:1}           & $\xi\in\I(X)$                                     & \hspace{-0.8em}
			$\begin{array}{l}\xi\in\I(X-X_0);\\[0.1ex]
				\xi\circ X=\xi\circ(X-X_0);\\[0.1ex]
				\Delta (\xi\circ X)=\xi(\Delta X)\medskip
			\end{array}$\\[1ex]
			P\ref{P:integral}  & $\zeta\in L(X)$\quad \  & \hspace{-0.8em}
			$\begin{array}{l}
				\zeta\id\in\I(X);\\[0.1ex]
				\zeta\cdot X = \zeta\id\circ X\medskip
			\end{array}$\\
			R\ref{R:200530}&$\zeta\in L(X)$; $\psi\in \I(\zeta\cdot X)$  & \hspace{-0.8em}
			$\begin{array}{l}
				\psi(\zeta\id)\in\I(X);\\[0.1ex]
				\psi\circ(\zeta\id\circ X) = \psi(\zeta\id)\circ X\medskip
			\end{array}$\\[1.4ex]
			C\ref{C:200522}& $\xi\in\I(X)$; $\psi\in\I(\xi\circ X)$; $\psi'(0)$ locally bounded\qquad\
		  & \hspace{-0.8em}
			$\begin{array}{l}
				\psi(\xi)\in\I(X);\\[0.1ex]
				\psi\circ(\xi\circ X)=\psi(\xi)\circ X\medskip
			\end{array}$\\[1.4ex]
			\hline\\
		\end{tabular}
		}\vspace{-2ex}
\end{table}

We shall say more on the benefits of the calculus in the concluding Section~\ref{S:6} once all notation has been introduced. The basic message is encouraging: with appropriate care one can hide much of the required stochastic analysis (stochastic integrals, jump-measure integrals) under the hood and treat common operations on stochastic processes \emph{algebraically}, as compositions of functions (indeed, $\xi\circ X$ can be interpreted in some cases as the $\xi$--variation of $X$; see Remark~\ref{R:predictable variation}). The benefits of doing so are significant, especially in the context of measure changes. 

One might expect the operations \eqref{intro1}--\eqref{intro3} to always work when the representing process $X$ is a pure-jump process of finite variation. Using only standard techniques, this intuition is false, however, because an integral of a finite variation semimartingale need not itself be of finite variation. We do obtain universal validity of rules \eqref{intro1}--\eqref{intro3} for pure-jump processes after suitably extending the standard integrals with respect to random measures. This universality then applies to all representing processes $X$ that belong \emph{sigma-locally} to the class of finite-variation pure-jump semimartingales; see Subsection~\ref{SS:sigma}.

The paper is organized as follows. Section~\ref{S:setup} introduces notation and reviews important concepts such as integration with respect to a complex-valued semimartingale.  
Section~\ref{S:3} defines representation of a semimartingale and derives important properties thereof, such as \eqref{intro1}--\eqref{intro3}. Here one gets to see an explicit formula for the object $\xi\circ X$, which is formulated in terms of real derivatives of the complex function $\xi$. This formula looks quite natural in the special case when both $\xi$ and $X$ are real-valued; such simplicity is also preserved when $\xi$ is analytic at 0 but this is much harder to see in the original definition. Subsection~\ref{SS:Wirtinger} provides an alternative form of the most general representation formula in terms of so-called Wirtinger derivatives, where the simplification in the analytic case is plainly visible.
Section~\ref{S:4} lists and proves a number of useful representations, among them generalizations of the Yor formula, thereby illustrating the strength of the proposed calculus. This section also provides counterexamples that document tightness of the results obtained in Section~\ref{S:3}. Section~\ref{S:5} summarizes the computation of predictable characteristics of a represented semimartingale. Finally, Section~\ref{S:6} discusses additional benefits of the proposed calculus and directions for future research.

\section{Setup and notation}\label{S:setup}
This section provides background on complex numbers and the probabilistic setup. It furthermore reviews stochastic integration for complex-valued semimartingales, the notion of predictable functions, and sigma-localized integrals with respect to random measures.

\subsection{The lift from \texorpdfstring{$\Cx$}{C} to \texorpdfstring{$\R^2$}{R2}}
 Below, we explicitly allow quantities to be complex-valued in order to allow for a consistent treatment of complex integrals, exponentials, etc., and in particular characteristic functions.  The reader interested only in real-valued calculus can easily skip this subsection and always replace the general `$\Cx$--valued' by the special case `$\R$--valued' in their mind.
Throughout this section, let $m \in \N$ denote an integer. To simplify notation later on, we write  $\Cinf^m =\Cx^m \cup \{\NaN\}$ for some `non-number' $\NaN \notin \bigcup_{k\in\N}\Cx^k$. 
We introduce the function $\id:\Cinf^m\to\Cinf^m$  by $\id(v)=v$.

The definitions below now hinge on the identification map $\hat \id:\Cinf\rightarrow \R^{2} \cup\{\NaN\}$ given by
\[
\hat \id(v)=\left[ 
\begin{array}{c}
\Re v \\ 
\Im v
\end{array}
\right],  \quad v \in \Cx;\qquad \hat \id(\NaN) = \NaN,
\]
and its appropriate multidimensional extension, again denoted by $\hat \id:\Cinf^m\rightarrow \R^{2m} \cup \{\NaN\}$ given by
$$\hat \id(v) = (\Re v_1,\Im v_1,\ldots,\Re v_m,\Im v_m)^\top,\quad v \in \Cx^m;\qquad \hat \id(\NaN) = \NaN.$$ 
Observe that $\hat \id(v) \in \R^{2m}$ for $v \in \Cx^m$ contains the values of $\Re v$ and $\Im v$, interlaced. 
At times we silently use matrix-valued versions of these canonical maps, which are taken to double the row dimension but which we do not introduce formally to avoid excessive notation. 

So as not to obscure the main ideas with notation, we will highlight the key properties of the lift $\hat \id$ for $m =1$.
To this end, the inverse map to $\hat \id$ is 
$\hat \id^{-1}:\R^2 \cup \{\NaN\} \rightarrow \Cinf$ given by
\[
\hat \id^{-1}\left([x\ y]^\top \right)=x+iy,  \quad [x\ y]^\top \in \R^2;\qquad \hat \id^{-1}(\NaN) = \NaN.
\]
The following two properties of $\hat \id$ are of importance:
\begin{itemize}
\item $\hat \id$ and $\hat \id^{-1}$ are linear, when restricted to $\Cx$ and $\R^2$;
\item for $u,v\in \mathbb{C}$ one obtains
\begin{equation}\label{eq:190109.2}
\hat \id(uv)=\left[\hat \id(u)\text{ }\hat \id(iu)\right]\hat \id(v)
=\left[ 
\begin{array}{cc}
\Re u & -\Im u \\ 
\Im u & \Re u
\end{array}
\right] \left[ 
\begin{array}{c}
\Re v \\ 
\Im v
\end{array}
\right] . 
\end{equation}
\end{itemize}

\subsection{Probabilistic quantities}\label{sect: technicalities}
We fix a probability space $(\Omega ,\sigalg{F},\P)$ with a right-continuous filtration $\filt{F}$. We shall assume, without loss of generality, that all semimartingales are right-continuous, and have left limits almost surely. For a brief review of standard results without the assumption that the filtration is augmented by null sets, see Perkowski and Ruf~\cite[Appendix~A]{Perkowski_Ruf_2014}. We follow mostly the notation of Jacod and Shiryaev~\cite{js.03}. 

For a $\Cx^m$--valued stochastic process $V$ we shall write $\hat V = \hat \id(V)$ for the corresponding $\R^{2m}$--valued process.  

\begin{definition}[Complex-valued process properties]\label{D:propertiesCx}
A $\mathbb{C}^{m}$--valued stochastic process $V$ is said to have a certain property, for example to be a
semimartingale (respectively, martingale; local martingale; special semimartingale; process of finite
variation; process with independent increments; predictable; locally bounded; etc.) if the $
\mathbb{R}^{2m}$--valued process $\hat V = \hat \id(V) $ has that same property, i.e., if 
$\hat V$ is a semimartingale (respectively, martingale, etc.).\qed
\end{definition}

We denote the left-limit process of a (complex-valued) semimartingale $V$ by  $V_-$ and use the convention $V_{0-} = V_0$. We also set $\Delta V = V - V_-$; in particular we have $\Delta V_0 = 0$. 
For complex-valued processes, the quadratic variation process is defined to be bilinear.%
\footnote{The bilinear definition is more prevalent. It is used, for example, in Dol\'eans-Dade~\cite{doleans-dade.70}, \'Emery~\cite{emery.89}, Revuz and Yor~\cite{revuz.yor.91}, and Protter~\cite{protter.05}. The sesquilinear alternative appears in Getoor and Sharpe~\cite{getoor.sharpe.72}.}
That is, for $\Cx$--valued semimartingales $V$ and $U$ we set  
\begin{equation*}
 [V, U]=[\Re V, \Re U]-[\Im V, \Im U] + i \left([\Re V,\Im U]+[\Im V,\Re U]\right).
\end{equation*}
We have again $[V, U]_0 = 0$.   If $V$ is $\Cx^m$--valued, then $[V, V]$ denotes the corresponding $\Cx^{m \times m}$--valued quadratic variation, formally given by
\begin{equation*}
 [V, V]=  (I_m \otimes [1\ i]) \left[\hat V, \hat V\right] \left(I_m \otimes \left[ 
 \begin{array}{c}
1 \\ 
i
\end{array}
\right]\right) ,
\end{equation*}
where $I_m$ denotes the $m\times m$ identity matrix and $\otimes$ the Kronecker product. Observe that for a $\Cx^{n \times m}$--valued matrix $R$ we have $[RV, RV] = R [V,V] R^\top$.
Furthermore, we write $[V, V]^c$ for the continuous part of the quadratic variation $[V, V]$ (the latter being of finite variation).

\begin{remark}[Alternative characterisations of {$[V,V]^c$}] \label{R:170813.1} 
We might call $V^c(\P)$ the continuous local martingale part of a semimartingale $V$; see \cite[I.4.27]{js.03}.  Note that $V^c(\P)$ depends on the under\-lying measure $\P$. To wit, for two equivalent measures $\Qu \sim \P$, we  usually have $V^c(\Qu) \neq V^c(\P)$ if $\Qu \neq \P$. Nevertheless, we always have $$[V, V]^c  = [V^c(\P), V^c(\P)] = [V^c(\Qu), V^c(\Qu)];$$ see also Dellacherie and Meyer~\cite[Theorem VIII.27]{dellacherie.meyer.82} and Protter~\cite[p.~70]{protter.05}.
\qed
\end{remark}

Let $\mu^V$ denote the jump measure of a semimartingale $V$ and $\nu^V$ its predictable compensator (under a fixed probability measure $\P$).  Then for a $\Cx$--valued  bounded predictable function $\xi$ (a precise definition is provided in Subsection~\ref{SS:predictable}) with $\xi(0) = 0$ we  have
\[	\xi *  \mu^V =  \xi \bigs(\hat \id^{-1}\bigs) *  \mu^{\hat V} = \sum_{t \leq \cdot} \xi_t(\Delta V_t),
\]
 provided $\abs{\xi} * \mu^V < \infty$.  Then  $\nu^V$ is a predictable random measure such that $\xi* \mu^V - \xi * \nu^V$ is a local martingale.  Observe
 furthermore that for an $m$--dimensional semimartingale $V$ we  have
\[
	[V,V]^c = [V,V] - \id\, \id^\top * \mu^{V}.
\]	

If $V$ is special, we let the triplet $(B^V, [\hat V,\hat V]^c, \nu^V)$ denote the  corresponding semimartingale characteristics of $V$ under a fixed probability measure $\P$.%
\footnote{We use the real-valued lift of $V$ to describe the continuous part of the quadratic variation in the characteristic triplet. This is necessary to capture the full dynamics of $V$. For example, let $V$ and $W$ denote two independent $\R$--valued Brownian motions and set $Z = \sqrt{2} V + i W$. Then $[V, V]^c = [Z,Z]^c$ but indeed  $[\hat V, \hat V]^c \neq  [\hat Z,\hat Z]^c$.}  
In particular,  the drift $B^V$, i.e., the predictable finite-variation part of the Doob--Meyer decomposition of $V$, is always assumed to start in zero, i.e., $B^V_0 = 0$.
For a general $m$--dimensional semimartingale $V$, we write  $V[1] = V - \id \indicator{\abs{\id} > 1} * \mu^V$. We can then define the `clock' (or `activity') process
\[
	A^V = \sum_{i = 1}^{2m} {\rm TV} \left(B^{\hat V[1]}_i\right) +  \tr \bigs[\hat V, \hat V\bigs]^c + (\abs{\id}^2 \wedge 1) * \nu^V,
\]
where $\rm TV$ denotes total variation. 
Then $A^V$ is  non-decreasing and locally bounded. Thanks to \cite[II.2.9]{js.03}, there exists an appropriate transition kernel $F^V$ such that 
$$\nu^V(\d t,\d v) = F^V(\d v)\d A^V_t. $$

\subsection{Stochastic integration}
In this subsection we discuss stochastic integrals of  predictable processes with respect to complex-valued semimartingales. To begin, consider a $\Cinf^{1 \times m}$--valued process $\zeta$ and a $\Cx^m$--valued semimartingale $V$.  Here $\zeta$ is explicitly allowed to take the value $\NaN$, but needs to be $\Cx^{1 \times m}$--valued, $(\P \times A^V)$--a.e., for the integral to be defined. 
If $V$ is real-valued, then we write $\zeta \in L(V)$ if both $\Re \zeta$ and $\Im \zeta$ are integrable with respect to $V$ (in the standard sense). We then set $\zeta \cdot V = (\Re \zeta) \cdot V + i (\Im \zeta) \cdot V$. 

If $V$ is complex-valued,  then we say $\zeta \in L(V)$ if $( \zeta\otimes [1\ i]) \in L( \hat V)$, where $\otimes $ represents the Kronecker product; recall also \eqref{eq:190109.2}. We then write 
\begin{equation}\label{eq:190109.1} 
	\zeta \cdot V = (\zeta\otimes [1\ i]) \cdot \hat V
\end{equation}
for the stochastic integral of $\zeta$ with respect to $V$. For real-valued $V$ the class $L(V)$ is defined twice but it is clear that the two definitions are consistent and $\zeta \cdot V$ is well defined. For $m=1$ one has $\zeta\in {L}(V) $ if and only if $[\zeta\ \ i\zeta]\in L(\hat V)$.   It is clear how to extend this definition to a $\Cinf^{n \times m}$--valued process $\zeta$, where $n \in \N$.

\begin{remark}[Caveat of complex-valued integration]
Complex-valued stochastic integrals appear in the literature in a very
limited context such as stochastic differential equations (e.g., \cite[I.4.60]{js.03}) or the It\^{o} formula 
(e.g., \cite[Proposition~V.2.3]{revuz.yor.91}). In those circumstances the integrands
are locally bounded, meaning that  vector-valued integration is not
required and integrability itself is not an issue. Our definition coincides
with these special cases when $\zeta$ is locally bounded but in
general the (real) stochastic integrals on the right-hand side of \eqref{eq:190109.1} 
cannot be computed component-wise.\qed
\end{remark}

Finally, for a $\Cx^{n \times m \times m}$--valued process $\zeta$ and a  $\Cx^{m \times m}$--valued
semimartingale $V$ (usually a quadratic variation process), let $\vecr(\zeta)$  and $\vecc(V)$ denote the row-wise and column-wise flattening of $\zeta$ and $V$, respectively. Then $\vecr(\zeta)$ is $(n \times m^2)$--dimensional and $\vecc(V)$ is $m^2$--dimensional. We then write $\zeta \in L(V)$ if $\vecr(\zeta) \in L(\vecc(V))$ and $\zeta \cdot V = \vecr(\zeta) \cdot \vecc(V)$.

\subsection{Predictable functions}\label{SS:predictable}
For this subsection, let $m,n \in \N$. 
As in \cite[II.1.4]{js.03}, we consider the notion of a predictable function on 
$\overline {\Omega}^m = \Omega \times [0,\infty) \times \Cinf^m$.
For two predictable functions $\xi: \overline{\Omega}^m \rightarrow \Cinf^n$ and $\psi: {\overline{\Omega}}^n \rightarrow \Cinf$  we shall write $\psi(\xi)$ to denote the function $(\omega, t, x) \mapsto \psi(\omega, t, \xi(\omega, t, x))$ with the convention
$\psi(\omega, t,  \NaN)=\NaN$.
If $\psi$ and $\xi$ are predictable, then so is $\psi(\xi)$.

For a predictable function $\xi: \overline{\Omega}^m \rightarrow \Cinf^n$ we shall write $\hat \xi = \hat \id (\xi)$ and $\xi^{(k)}$ for the $k$--th component of $\xi$, where $k \in \{1, \cdots, n\}$. We also write $\hat D \xi$ and $\hat D^2 \xi$ for the real derivatives of $\xi$, i.e., $\hat D_i \xi^{(k)}$ is the composition of the $i$--th element of the gradient of $\xi^{(k)} \bigs( \hat \id^{-1}\bigs)$ and  the lift $\hat \id$ and $\hat D_{i,j}^2  \xi^{(k)}$ is the composition of the $(i,j)$--th element of the Hessian of $\xi^{(k)} \bigs( \hat \id^{-1}\bigs)$ and the lift $\hat \id$, for $i,j \in \{1, \cdots, 2 m\}$.
Note that $\hat D \xi$ has dimension $n \times (2m)$, $\hat D^2 \xi$ has dimension $n \times (2m) \times (2m)$, and the domains of $\hat D \xi$, $\hat D^2 \xi$ equal $\overline {\Omega}^m$, i.e., they coincide with the domain of $\xi$.

We want  to allow for predictable functions such as $\xi=\log(1+\id)$ whose effective domain is not the entire $\Cx$. For this reason, we define, for a given predictable function $\xi: \overline{\Omega}^m \rightarrow \Cinf^n$, the set of semimartingales whose jumps are compatible with $\xi$, i.e.,
\begin{equation*}
\Dom(\xi) = \left\{\text{semimartingale } V:\xi(\Delta V) \text{ is $\Cx^n$--valued}, \P\text{--almost surely}\right\}.
\end{equation*}
If for another predictable function $\psi: \overline{\Omega}^n \rightarrow \Cinf^m$ we have
 $\psi (\xi(\Delta V)) = \Delta V$ for all $V\in \Dom(\xi)$, we say $\xi$ allows for a left inverse.
 If $\xi(\psi(\Delta V)) = \Delta V$ for all $V \in \Dom(\psi)$ we say that $\xi$ allows for a right inverse. If $\psi$ represents both left and right inverse we shall use the notation $\xi^{-1} = \psi$.

\subsection{Sigma-localized integrals with respect to random measures} \label{SS:sigma}
We next recall from \v{C}ern\'{y} and Ruf~\cite{cr0} relevant results about the sigma-localized version of the $*$ integral of a predictable function with respect to $\nu^V$ and $\mu^V$ for a semimartingale $V$, which we fix from now on to the end of this section. The following is adapted from \cite[Definition~3.1]{cr0}.
\begin{definition}[Extended integral with respect to random measure] \label{D:190405}
 Denote by $L(\mu^V)$ the set of predictable functions that are absolutely integrable with respect to $\mu^V$. We say that a predictable function $\xi$ belongs to $L_\sigma(\mu^V)$, the 
sigma--localized class of $L(\mu^V)$, 
if there is a sequence $(C_k)_{k \in \N}$ of predictable sets increasing to 
$\Omega\times[0,\infty)$ and a semimartingale $Y$ such that $\indicator{C_k}\xi\in L(\mu^V)$ for each $k\in\N$ and 
$$ (\indicator{C_k}\xi)*\mu^V = \indicator{C_k}\cdot Y,\qquad k\in\N.$$
In such case the semimartingale $Y$ is denoted by $\xi \star \mu^V$.

Similarly, we define  $L_\sigma(\nu^V)$ and $\xi \star \nu^V$. \qed
\end{definition}

In the following, we recall useful characterizations for $L_\sigma(\nu^V)$ and $L_\sigma(\mu^V)$.
\begin{proposition}[Kallsen~\cite{kallsen.04}, Definition~4.1, Lemma~4.1]
For a predictable function $\xi$ the following statements are equivalent.
\begin{enumerate}[label={\rm(\roman{*})}, ref={\rm(\roman{*})}]  \label{L:190405}
\item $\xi\in L_\sigma(\nu^V)$.
\item The following two conditions hold:
	\begin{enumerate}[label={\rm(\alph{*})}, ref={\rm(\alph{*})}]
		\item $\int\abs{\xi_t(v)} F^V_t(\d v) < \infty \quad (\P \times A^V)$--a.e. 
		\item $\int_0^\cdot \left| \int \xi_t(v) F^V_t(\d v) \right|  \d A^V_t < \infty$.
	\end{enumerate}
\end{enumerate}
Moreover, for $\xi\in L_{\sigma}(\nu^V)$ one has 
\[
\xi \star \nu^V =\int_0^\cdot \left(\int \xi_t( v) F^V_t(\d v) \right) \d A_t^V.
\]
\end{proposition}

\begin{proposition}[\cite{cr0}, Proposition~3.4]\label{P:190330}
For a predictable function $\xi$ the following statements are equivalent.
\begin{enumerate}[label={\rm(\roman{*})}, ref={\rm(\roman{*})}]
\item\label{P:190330.i} $\xi\in L_\sigma(\mu^V)$.
\item\label{P:190330.ii} The following two conditions hold.
		\begin{enumerate}[label={\rm(\alph{*})}, ref={\rm(\alph{*})}]
				\item $\abs{\xi}^2 * \mu^V < \infty$.
				\item $\xi \indicator{\{\abs{\xi} \leq 1\}} \in L_\sigma(\nu^V)$. 
			\end{enumerate}
\end{enumerate}
Furthermore, for $\xi \in L_{\sigma}(\mu^V)$ one has 
\begin{equation} \label{eq:190405.1}
\xi \star \mu^V =\xi\indicator{\{\abs{\xi}> 1\}}*\mu^V+\xi\indicator{\{\abs{\xi}\leq 1\}}*(\mu^V-\nu^V)+\xi\indicator{\{\abs{\xi}\leq 1\}}\star\nu^V,
\end{equation}
where the integral with respect to $\mu^V-\nu^V$ is defined in \cite[II.1.27(b)]{js.03}.
\end{proposition}

\begin{remark}[\cite{cr0}, Remarks~3.2 and 3.5]\ \label{R:190405}
	Let $\Qu$ denote a probability measure absolutely continuous  with respect to $\P$. With the obvious notation, we then have  $L_\sigma^\P(\mu^V) \subset L_\sigma^\Qu(\mu^V)$. For $L_\sigma^\P(\nu^V(\P))$ and $L_\sigma^\Qu(\nu^V(\Qu))$, no such inclusions hold in general. However, for $\xi$ with $\xi^2 * \mu^V < \infty$ Proposition~\ref{P:190330} yields that if $\xi \indicator{\{\abs{\xi} \leq 1\}} \in L_\sigma(\nu^V(\P))$ then also $\xi \indicator{\{\abs{\xi} \leq 1\}} \in L_\sigma(\nu^V(\Qu))$.
\qed
\end{remark}

 Next we recall a composition property for stochastic integrals. Such result does not hold if the $\star$ integral were to be replaced by the $*$ integral.
\begin{proposition}[\cite{cr0}, Proposition~3.9]\label{P:190411}
For $\xi \in L_\sigma(\mu^V)$ taking values in $\Cinf^n$ for some $n \in \N$ and a $\Cinf^{1, n}$--valued predictable process $\zeta$ the following statements are equivalent.

\begin{enumerate}[label={\rm(\roman{*})}, ref={\rm(\roman{*})}]
\item\label{P:190411.i} $\zeta\in L(\xi\star\mu^V)$.
\item\label{P:190411.ii} $\zeta\xi\in L_\sigma(\mu^V)$.
\end{enumerate}
Furthermore, if either condition holds then
$\zeta\cdot(\xi\star\mu^V) = (\zeta\xi)\star\mu^V$.
\end{proposition}

The previous three propositions and the remaining ones of this section are proved in the corresponding references for the case when $V$ is $\R$--valued. The arguments for the general case are straightforward; see also \cite[Remark~2.1]{cr0}.

We next denote by $\V$ the set of semimartingales with finite variation on compact time intervals and by $\V^{\d}$ the subset of finite variation pure-jump processes, i.e., those semimartingales $V\in\V$ that satisfy $V = V_0 + \id \ast \mu^V$.
The statements in  this subsection can also be expressed in terms of a special class of semimartingales $\V^{\d}_\sigma$, i.e., the  $\sigma$--localized class of finite variation pure-jump processes.  The key connection is the following.

\begin{proposition}[\cite{cr0}, Proposition~3.12]\label{P:190911}
If $\xi \in L_\sigma(\mu^V)$ then $\xi \star \mu^V$ is an element of $\V^{\d}_\sigma$. Conversely, if $V\in \V^{\d}_\sigma$ then $\id \in L_\sigma(\mu^V)$ and $V = V_0 + \id \star \mu^V$.
\end{proposition}

We conclude this section with a natural decomposition of $V$ into jumps at predictable times and a quasi-left-continuous process. 
\begin{proposition}[\cite{cr0}, Proposition~3.15]  \label{P:190729}
Every semimartingale $V$ has the unique decomposition 
\begin{equation*}
V = V_0 + V^{\q} + V^{\ddp},
\end{equation*}
where  $V^\q_0=V^{\ddp}_0 = 0$, $V^{\q}$ is a quasi-left-continuous semimartingale,  $V^{\ddp}$ jumps only at predictable times, and $V^{\ddp} \in \V^{\d}_\sigma$. We then have $ [V^\q,V^\ddp] = 0$.
\end{proposition}
If we define the predictable set
$	\H_V = \left\{\nu^V(\{\cdot\})=0\right\}$,
then indeed $V^{\q} = \indicator{\H_V} \cdot V$ and $V^{\ddp} = \indicator{\H_V^c}\cdot V$. Hence $V$ is special if and only if both $V^\q$ and $V^\ddp$ are special.

Let $\mathcal{T}_V$ denote a countable family of stopping times that exhausts the jumps of $V^\ddp$.\footnote{Note that $\P[\Delta V_\tau = 0] > 0$ is possible for $\tau \in \mathcal T_V$.} For each $V$ there may be many ways to choose $\mathcal{T}_V$. The following statement holds for any such $\mathcal{T}_V$.

\begin{proposition}[Drift of a pure-jump process jumping only at predictable times]\label{P:190520}
Assume that  $V^\ddp$ is special. Then we have
\begin{equation*}
B^{V^\ddp}  = \sum_{\tau\in\mathcal{T}_V}  \E_{\tau-}[\Delta V_\tau]\indicator{\lc\tau,\infty \lc}.
\end{equation*}
\end{proposition}
\begin{proof}
Thanks to \eqref{eq:190405.1}, we have $B^{V^\ddp} = \id  \star \nu^{V^\ddp}$. Moreover, $B^{V^\ddp}$ is of sigma-finite variation and $\mathcal T_V$ exhausts its jumps. Proposition~4.6 in \cite{cr0} applied to $B^{V^\ddp}$ then yields the result.
\end{proof}

\section{Semimartingale representation} \label{S:3}
\subsection{Definition and basic properties}
From now on we shall fix some $d, n\in \N$ and consider a $\Cx^d$--valued semimartingale $X$.   We shall then study a variety of predictable transformations of $X$. Of course, an $\R^d$--valued semimartingale can always be considered a special case.

\begin{example}[A motivational example] \label{ex:181106}
	Let $X$ denote an $\R$--valued semimartingale and let $f: \R \rightarrow \R$ denote a twice continuously differentiable function.  Then it is well known that also the process $Y = f(X)$  is a semimartingale. More precisely, the It\^o--Meyer change of variables formula, \cite[I.4.57]{js.03},  provides the representation
\begin{equation}\label{eq:170702}
		Y =  f(X_0) + f'(X_{-}) \cdot X + \frac{1}{2} f''(X_{-}) \cdot [X,X]^c 
		            + \left(f(X_{-}  + \id) - f(X_{-} ) - f'(X_{-})\,\id\right) * \mu^X. 
\end{equation}

Let us now introduce the predictable function $\xi^{f, X}: \Omega \times [0,\infty) \times \R \rightarrow \R$ by
\begin{align*}
	\xi^{f, X}(\omega, t, x) =  f(X_{t-}(\omega)  + x) - f(X_{t-}(\omega)).
\end{align*}
Note that the derivatives $D\xi^{f, X}$ and $D^2 \xi^{f, X}$ exist. 
The representation in \eqref{eq:170702} then can be written in the more compact form
	\begin{align} \label{eq:170704}
		Y &= Y_0 +  D\xi^{f, X}(0) \cdot X+ \frac{1}{2}D^2\xi^{f, X}(0) \cdot [X,X]^c + \left(\xi^{f, X}- D\xi^{f, X}(0)\, \id \right) * \mu^X. 
	\end{align}	
Observe that $\Delta Y=\xi(\Delta X)$ and that $Y$ is fully determined by $X$ and the predictable function $\xi^{f, X}$.
\qed
\end{example}
The connection between \eqref{eq:170702} and \eqref{eq:170704} motivates the key concept of this paper, Definition~\ref{D:181101} below. Recall from 
Subsection~\ref{SS:sigma} the predictable set $\H_X$, on which $X^\ddp$ has no `activity.'

\begin{definition}[Representing functions for a given semimartingale $X$]  \label{D:170714}
Let $\I^{n}(X)$ denote the set of all predictable functions $\xi: \widebar{\Omega}^d \rightarrow \Cinf^n$  such that the following properties hold.
	\begin{enumerate}[label={\rm(\arabic{*})}, ref={\rm(\arabic{*})}]
		\item\label{D_I(X):1}  $X\in \Dom(\xi)$, viewed as a property of $\xi$ for fixed $X$.
		\item\label{D_I(X):2}  $\xi(0) = 0$, $(\P \times  A^X)$--a.e.
		\item\label{D_I(X):3} $x \mapsto \indicator{\H_X} \xi( x)$ is twice real-differentiable at zero, $(\P \times  A^X)$--a.e.
		\item\label{D_I(X):4} $\indicator{\H_X} \hat D \xi(0) \in L(\hat X)$.
		\item\label{D_I(X):5} $\hat D^2 \xi(0) \in L([\hat{X},\hat{X}]^c)$.
		\item\label{D_I(X):6} $(\xi -   \indicator{\mathcal H_X} \hat  D  \xi(0) \, \hat \id) \in L_\sigma(\mu^{{X}})$.
	\end{enumerate}
We write $\I(X)=\bigcup_{k\in\N}\I^k(X)$.\qed
\end{definition}

\begin{remark}[The role of the predictable set $\H_X$] \label{R:200518.1}
	If a predictable function $\xi$ satisfies the conditions of Definition~\ref{D:170714} with $\H_X$ replaced by a larger predictable set $\H \supset \H_X$ (e.g., $\H = \Omega \times [0, \infty)$, corresponding to no indicators at all), then the conclusion $\xi \in \I(X)$ still holds.  To see this, we only need to argue \ref{D_I(X):6}.  This follows from observing that we have
	$\indicator{\H \setminus \H_X} \hat D \xi(0) \in L(\hat X)$, yielding $\indicator{\H \setminus \H_X}  \hat D \xi(0) \,\hat \id \in L_\sigma(\mu^{\hat{X}})$ by Proposition~\ref{P:190411}.
	
	 Example~\ref{E:190916} below provides an instance where $X = X^\ddp$, $\xi \in \I(X)$,   $\xi$ is  twice differentiable at zero, but $D \xi(0) \notin L(X)$. Thus, allowing for the existence of an appropriate predictable set $\mathcal H_X$ such that only $\indicator{\mathcal H_X} D \xi(0) \in L(X)$ is required, indeed allows for a bigger class $\I(X)$.
	\qed
\end{remark}

As Propositions~\ref{P: universal} and \ref{P:Ito} and Theorem~\ref{T:composition0} below argue, the following class $\Uni$ enjoys closedness with respect to common operations and universality in the sense that a representing function $\xi\in\Uni$ satisfies $\xi\in\I(X)$ for \emph{any} semimartingale $X$ provided that $\xi(\Delta X)$ is finite.
\begin{definition}[Universal representing functions] \label{D:210911}
Let $\Uni^{n}$ denote the set of all predictable functions $\xi: \widebar{\Omega}^d \rightarrow \Cinf^n$ such that the following properties hold, $\P$--almost surely.
\begin{enumerate}[label={\rm(\arabic{*})}, ref={\rm(\arabic{*})}]
			\item\label{I0bis:i}   $\xi_t(0) = 0$, for all $t \geq 0$.
			\item\label{I0bis:ii}  $x \mapsto \xi_t(x)$ is twice real-differentiable at zero, 	for all $t \geq 0$.
			\item\label{I0bis:iii} $\hat D \xi(0)$ and $\hat D^2 \xi(0)$ are locally bounded.
			\item\label{I0bis:iv} There is a predictable locally bounded process $K>0$ such that 
	$$\sup_{\abs{x}\leq \sfrac{1}{K}} \frac{\mathopen{\bigs|} \xi(x)-\hat D \xi(0)\hat{\id}(x)\mathclose{\bigs|}}{\abs{x}^2}\indicator{x\neq 0} \text{ is locally bounded.}$$
	\end{enumerate}
We write $\Uni=\bigcup_{n\in\N}\Uni^{n}$.\qed
\end{definition}

\begin{remark}[A special case: real-valued semimartingales]
If $X$ is real-valued then we may consider $\xi$ as a predictable function with real domain. In this case, it can be easily checked that in Definitions~\ref{D:170714} and \ref{D:210911} we may omit the  hats on top of $D$, $\id$, and $X$, with $D$ and $D^2$ being the standard gradient and Hessian, respectively.
	\qed
\end{remark}

\begin{proposition}[Universality of $\Uni$]\label{P: universal}
Fix some $\xi \in \Uni$ such that $X\in\Dom(\xi)$. We then have $\xi \in \I^n(X)$, $(\xi - \hat D \xi(0) \hat \id )\in L(\mu^X)$, and
$$ \left(\xi - \hat D \xi(0) \hat \id \right) \star \mu^{X} = \left(\xi - \hat D \xi(0) \hat \id \right) * \mu^{X}.$$
\end{proposition}

\begin{proof}
	The first claim follows from Remark~\ref{R:200518.1}. For the second claim it suffices to observe that $\abs{\xi - \hat D \xi(0) \hat \id} * \mu^{X} < \infty$ by localization.
\end{proof}

\begin{proposition}[Properties of $\I(X)$] \label{P:1_J}
The following statements hold.
\begin{enumerate}[label={\rm(\arabic{*})}, ref={\rm(\arabic{*})}]
\item\label{P:1_J:i} If  $\xi, \psi \in \I^n(X)$ for some $n \in \N$ and $\lambda \in \Cx$ then $\xi + \lambda \psi \in \I^n(X)$.
\item\label{P:1_J:ii} If $X \in \V^{\d}_\sigma$  then  $\I(X) \subset L_\sigma(\mu^X)$.  Moreover, if $X = X^\ddp$ then  $\I(X) = L_\sigma(\mu^X)$.
 \item\label{P:1_J:iii} Let $\mathcal{H}$ denote a predictable set. Then  
$\I(X)=\I(\indicator{\mathcal H} \cdot X)\cap\I(\indicator{\mathcal H^c} \cdot X)$; in particular, 
$\I(X) = \I(X^\q) \cap \I(X^\ddp)$.
 \item\label{P:1_J:iv}  Let  $Y$ denote another semimartingale  and let $\psi \in \I(X,Y)$. If $\psi$ is constant in the $y$--argument then $\xi: x \mapsto \psi(x,0)$ is in $\I(X)$.
\end{enumerate}

\end{proposition}
\begin{proof}
Parts~\ref{P:1_J:i} and \ref{P:1_J:iv} follow directly from Definition~\ref{D:170714}. Parts~\ref{P:1_J:ii} and \ref{P:1_J:iii}  rely on  an application of Proposition~\ref{P:190411}.  
\end{proof}

\begin{definition}[Semimartingale representation]  \label{D:181101}
For a predictable function $\xi \in \I(X)$	we use the notation 
	\begin{align}  \label{eq:190610.2}
	\xi\circ X&= \indicator{\H_X}   \hat D \xi (0) \cdot \hat X+ \frac{1}{2} \hat D^2 \xi(0) \cdot \bigs[\hat X,\hat X\bigs]^c + \left(\xi - \indicator{\H_X}   \hat D \xi(0) \, \hat \id  \right) \star \mu^{X}.
	\end{align}	
If there exists $\xi \in \I(X)$ such that 
	\begin{align} \label{eq:170704.4}
		Y &=  Y_0 + \xi \circ X,
	\end{align}	 
we say that the semimartingale $Y$ is represented in terms of the semimartingale $X$. \qed
\end{definition}

\begin{remark}[\'Emery formula]
The right-hand side of equation \eqref{eq:190610.2} appears almost verbatim in \'Emery~\cite[eq.~(13)]{emery.78} in the special case where $X$ is real-valued and $\xi$ is a real-valued twice continuously differentiable time-constant and deterministic function; see also Proposition~\ref{P: universal}. In this case the $\star$ integral can be replaced by the standard $*$ integral. \qed
\end{remark}

\begin{remark}[Interpretation of $\xi\circ X$ as $\xi$--variation]\label{R:predictable variation}
The object $\xi\circ X$ with time-constant deterministic $\xi$, most often a power function, resurfaces several times in the literature under the name $\xi$--variation, see Dol\'eans~\cite{doleans.69}, Monroe~\cite{monroe.72,monroe.76}, L\'epingle~\cite{lepingle.76}, Jacod~\cite{jacod.08}, and Carr and Lee~\cite{carr.lee.13}. The terminology and \'Emery's \cite{emery.78} notation $\int_0^\cdot \xi(\d X_s)$ originate from the fact that, for suitably regular time-constant deterministic $\xi$, the partial sums 
$ \sum_{n\in\N}\xi_{t_{n-1}} \left(X^{t_n}-X^{t_{n-1}}\right)$ 
converge uniformly on compact time intervals in probability to $\xi\circ X$ as the time partition $(t_n)_{n\in\N}$ becomes finer; see \cite[Th\'eor\`eme~2a]{emery.78} and \cite[Theorem~2.2]{jacod.08} for a related statement. For a precise statement of such convergence for predictable $\xi$, see \cite{crV}.\qed
\end{remark}

\begin{remark}[Generalizations of \'Emery formula] \label{R:200521}
The conditions in Definition~\ref{D:170714} ensure that all terms in \eqref{eq:170704.4} are defined.  One could extend the class $\I(X)$ further. The idea of such generalisation would be to focus  on the activity of the individual components of $X$. For example, one could abstain from requesting that $x \mapsto \indicator{\H_X} \xi(x)$ is real-differentiable in the $i$--th component for times when 
$\d A^{X^{(i)}} = 0$.  Moreover, one could assume that the second real derivative of $x \mapsto \indicator{\H_X} \xi(x)$ only needs to exist $(\P \times \tr [\hat X, \hat X]^c)$--a.e. However, such  generalisations would come with more complicated notation and would obscure the main results, hence we do not pursue them here.
\qed
\end{remark}

\begin{remark}[Measure invariance of representations]\label{R:181106}
Note that $\I(X)$ is invariant under equivalent changes of measures. More precisely, if $\Qu$ is a probability measure absolutely continuous with respect to $\P$ and if $\xi \in  \I(X)$ under $\P$, then also $\xi \in \I(X)$ under $\Qu$ (recall Remark~\ref{R:190405} to see this). Moreover, if we define $Y = \xi \circ X$ under $\P$, then we also have $Y = \xi \circ X$ under $\Qu$. Hence,  $\xi\circ X$ is  measure-invariant in the sense that \eqref{eq:190610.2}  only depends on the null sets. A similar statement holds for $\Uni$. This is in contrast to the common (and frequently also very useful) representation of $Y$ in terms of predictable characteristics.  
\qed
\end{remark}

We now list some immediate consequences of the definition of representability.

\begin{proposition}[Properties of representation]\label{P:1}
The following statements hold.
\begin{enumerate}[label={\rm(\arabic{*})}, ref={\rm(\arabic{*})}]
	\item\label{P:1:i} Let $\xi\in\I(X)$.
	Then
	\begin{align*}
		 \Delta (\xi \circ X) &=\xi(\Delta X). 
	\end{align*}
  \item\label{P:1:ii} $\I(X)=\I(X-X_0)$ and for any $\xi\in\I(X)$ one has 
	$$\xi\circ X = \xi\circ (X-X_0).$$ 
	\item\label{P:1:iii}  We have $\id_i, \id_i \id_j  \in \Uni^1$, for all $i,j \in \{1, \cdots, d\}$, with
		\begin{align*}
			X^{(i)}        &= X^{(i)}_0+\id_i\circ X;\\
			\bigs[X^{(i)}, X^{(j)}\bigs] &= (\id_i \id_j)\circ X.
		\end{align*}
	\item\label{P:1:iv}  If $\xi \in \I(X)$ then $\xi^* \in \I(X)$ and $(\xi \circ X)^* = \xi^* \circ X$, where  the superscript $*$ denotes the complex conjugate. 
	\item\label{P:1:v}  If $\xi \in \I(X)$ then $\xi \circ X^\q = (\xi \circ X)^\q$ and $\xi \circ X^\ddp = (\xi \circ X)^\ddp$. (Recall also Proposition~\ref{P:1_J}\ref{P:1_J:iii}).
	\item\label{P:1:vi} Let $Y$ be a predictable semimartingale of finite variation and $\xi \in \I(X,Y)$ such that 
	$\xi(\cdot, \Delta Y) \in \I(X)$ and 
	$\xi(0, \cdot) \in \I(Y)$. Then we have 
	\begin{align} \label{eq:200616}
		\xi \circ (X,Y)=\xi(\cdot, \Delta Y) \circ X + \xi(0, \cdot) \circ Y.
	\end{align}
\end{enumerate}
\end{proposition}
\begin{proof}
Parts \ref{P:1:i}, \ref{P:1:ii},  \ref{P:1:iv}, and \ref{P:1:v} follow directly from  Definitions~\ref{D:170714} and \ref{D:181101}.
Part~\ref{P:1:iii} follows directly from Proposition~\ref{P: universal} and Definition~\ref{D:181101}.  

For \ref{P:1:vi}, note that  $\xi(\cdot, \Delta Y) \in \I(X)$ yields that $\indicator{\H_{X,Y}}   \hat D_x \xi (0, 0) \in L(\hat X)$ and
$\hat D^2_{xx} \xi(0,0) \in L([X,X]^c)$ with
\begin{align*}
	\indicator{\H_X}   \hat D_x \xi (0, \Delta Y) \cdot \hat X &=
		 \indicator{\H_{X,Y}}   \hat D_x \xi (0, 0) \cdot \hat X
	;\\ 
	\hat D^2_{xx} \xi(0, \Delta Y) \cdot \bigs[\hat X,\hat X\bigs]^c  &= \hat D^2_{xx} \xi(0,0) \cdot \bigs[\hat X,\hat X\bigs]^c.
\end{align*}
Similarly 
$\xi(0, \cdot) \in \I(Y)$ yields that $\indicator{\H_{X,Y}}   \hat D_y \xi (0, 0) \in L(\hat Y)$  with
\begin{align*}
	\indicator{\H_Y}   \hat D_y \xi (0, 0) \cdot \hat Y &=
		 \indicator{\H_{X,Y}}   \hat D_y \xi (0, 0) \cdot \hat Y.
\end{align*}
Now, the result follows by comparing the jumps on the left and right hand side of \eqref{eq:200616}, for example by using \ref{P:1:i}.
\end{proof}

\begin{proposition}[Representation of stochastic integrals]\label{P:integral}
Let $\zeta$ be a $\mathbb{C}^{1\times d}$--valued predictable process in $L(X)$. Then $\zeta\id \in \I^1(X)$ and 
\[
\zeta\cdot X= \zeta\id \circ X.
\]
\end{proposition}
\begin{proof}
Let $I_d$ be a $d\times d$ identity matrix. Observe that $\xi =  \zeta\id$ verifies 
$$ \hat D\xi =\zeta\otimes[1\ i];
\qquad \hat D^2\xi =0;
\qquad \xi -  \hat D\xi(0) \,\hat \id = \zeta \left(\id -(I_d\otimes [1\ i])\, \hat\id\right) = 0.$$
Hence, $\xi$ belongs to $\I(X)$ as per Definition~\ref{D:170714}, and \eqref{eq:190610.2} together with \eqref{eq:190109.1} yield the claim.
\end{proof}

\begin{proposition}[Representation of a change of variables]\label{P:Ito}
Let $\U \subset \Cx^d$ be an open set such that $X_-, X \in \U$ and let $f: \U \rightarrow \Cx^n$ be  twice continuously real-differentiable.  
Then the predictable function $\xi^{f, X}:\widebar{\Omega}^d\to\Cinf^n$ defined by
\begin{align*}
	 \xi^{f,X}(x) = 
		\begin{cases}
			f\left(X_- + x\right) - f\left(X_-\right), &\quad X_- + x \in \U\\
			\NaN, &\quad X_- + x \notin \U
		\end{cases},
		\qquad x\in\Cx^d,
\end{align*}
belongs to $\Uni^{n}$,  $X\in\Dom(\xi^{f,X})$, and 
\[\pushQED{\qed}
f(X)=f(X_{0})+\xi^{f,X} \circ X.
\qedhere\popQED\]
\end{proposition}

\begin{proof}
Denote by $R > 0$ the distance from $X_-$ to the boundary of $\mathcal{U}$, by $R^*$ its running infimum, and by $\tau>0$  
the first time $R^*$ hits zero. The left-continuity of $R$ now yields $\tau=\infty$ and $R^*>0$. Therefore, $(\tau_n)_{n \in \N}$ given by $\tau_n=\inf\{t:R^*_t\le 1/n\}$, is a localizing sequence of stopping times that makes both $K=\sfrac{2}{R}$ and 
$ \sup_{\abs{x}\leq \sfrac{1}{K}} \mathopen{\bigs|}\hat D^2  \xi( x)\mathclose{\bigs|}$
locally bounded, yielding $\xi^{f,X}\in\Uni$.

Since $X \in \Dom(\xi^{f,X})$, Proposition~\ref{P: universal} now yields that 
\[
	f \bigs(\hat \id^{-1}\bigs(\hat X\bigs)\bigs) = f \bigs(\hat \id^{-1}\bigs(\hat X_0\bigs)\bigs)+  \xi^{f, X} \circ  X
\]
 is the  It\^o--Meyer change of variables formula for the real-valued function $\hat\id\bigs(f \bigs(\hat \id^{-1}\bigs)\bigs)$ applied to the real-valued process $\hat X$; see \cite[I.4.57]{js.03}. In view of $f \bigs(\hat \id^{-1}\bigs(\hat X\bigs)\bigs) = f(X)$  the proof is complete.
\end{proof}

\begin{remark}[It\^o's formula requires smoothness]
	It is possible to exhibit a function  $f: \R \rightarrow \R$ and a semimartingale $X$ such that $\xi^{f, X} \in \I(X)$, in the notation of Proposition~\ref{P:Ito}, and such that $f(X)$ is a semimartingale, but 
	$$f(X) \neq f(X_0) + \xi^{f, X} \circ X.$$ 
	For example, choose $X$ equal to Brownian motion started at 0 and $f = \abs{\id}$.  Here $f$ is not twice differentiable but $\xi^{f, X} \in \I(X)$ anyway as it is Lebesgue-a.e.~twice differentiable. Then $\xi^{f, X} \circ X = \sgn(X)\cdot X$ is another Brownian motion while $f(X)-f(X_0)$ is the absolute value of $X$.
\qed
\end{remark}

\subsection{Composition of representations}

We now describe the composition of  representations. It is this result along with its consequences that makes the calculus simple. 
\begin{theorem}[Composition of universal representing functions] \label{T:composition0}
The class $\Uni$ is closed under (dimensionally correct) composition, i.e., if $\xi \in\Uni^n$ and $\psi:  \widebar{\Omega}^n \rightarrow \Cinf$ is another predictable function with $\psi \in \Uni$, then $\psi(\xi)\in\Uni$. Furthermore, if $\psi(\xi(\Delta X))$ is finite-valued, then one has $\psi,\psi(\xi)\in\I(X)$, $\psi\in\I(\xi\circ X)$, and
$$\psi\circ(\xi\circ X) = \psi(\xi)\circ X.$$ 
\end{theorem}
\begin{proof}
Properties \ref{I0bis:i}--\ref{I0bis:iii} of Definition~\ref{D:210911} follow easily by direct calculation; see also \eqref{eq:200521.2} and \eqref{eq:200521.3} below. To show property \ref{I0bis:iv}, by localization we may assume that $\hat D \xi(0)$ is bounded and that there exists a constant $K_{\xi }>0$
such that 
$$\sup_{0 <\abs{x} \leq \sfrac{1}{K_\xi}} \frac{\abs{\xi(x)-\hat D \xi(0)\hat{\id}(x)}}{\abs{x}^2}<\infty.$$ 
An analogous statement applies to $\psi$, with some constant $K_{\psi }>0$. By possibly making $K_\xi$ larger we may also assume that $\sup_{\abs{x} \leq \sfrac{1}{K_\xi}}  \abs{\xi (x)} \leq \sfrac{1}{K_\psi}$.

For $\eta = \psi(\xi)$ we then have $\hat D\eta(0) = \hat D \psi(0) \hat D\hat\xi(0)$ and by the triangular inequality
\begin{align*}
\frac{\abs{\eta(x)-\hat D\eta(0) \hat\id(x)}}{\abs{x}^2}\indicator{x\neq 0}\leq{}& 
\frac{\abs{\psi(\xi(x))-\hat D\psi(0)\hat\id(\xi(x))}}{\abs{\xi(x)}^2}\indicator{\xi(x)\neq 0}\frac{\abs{\xi(x)}^2}{\abs{x}^2}\indicator{x\neq 0}\\
&{}+\frac{\abs{\hat D\psi(0)(\hat\xi(x)-\hat D \hat\xi(0)\hat{\id}(x))}}{\abs{x}^2}\indicator{x\neq 0}.
\end{align*}
In view of the boundedness of $\hat D \psi(0)$ and $\hat D\xi(0)$, the supremum on the right-hand side over $\abs{x}<\sfrac{1}{K_\xi}$ is finite. The statement $\psi(\xi)\in\Uni$ follows. By Proposition~\ref{P: universal} we have $\xi\in\I(X)$ and $\psi\in\I(\xi\circ X)$. As $\hat D\psi(0)$ is locally bounded, the rest of the statement follows from Theorem~\ref{T:composition} below.
\end{proof}
When $\xi$ and $\psi$ are deterministic and time-constant functions, Theorem~\ref{T:composition0} reduces to the tower property in Carr and Lee~\cite[Proposition 2.4]{carr.lee.13}.

 \begin{theorem}[Composition of semimartingale representations] \label{T:composition}
	Let $\xi \in \I(X)$. Moreover, fix $\psi \in \I(\xi\circ X)$ 
 such that $\psi(0) = 0$ and $\indicator{\H_X} \psi$ is twice real-differentiable at zero, $(\P \times A^{X})$--a.e., and 
					\begin{align} \label{eq:200521.1}	
					\indicator{\H_X} \hat D\psi(0) \in L\left(\hat D^2 \hat \xi(0) \cdot \bigs[\hat X,\hat X\bigs]^c\right)\cap
						L\left(\left(\hat \xi - \indicator{\H_X}   \hat D \hat \xi(0) \, \hat \id\right) \star \mu^X\right).
				\end{align}
		Then $\psi(\xi) \in \I(X)$ and we have
	\begin{equation}\label{composition}
		\psi \circ (\xi \circ X) = \psi(\xi) \circ X.
	\end{equation}
\end{theorem}
\begin{proof} Let $Y=\xi\circ X$. Without loss of generality, we may assume $\psi \in \I^1(Y)$. We need to check the six properties of Definition~\ref{D:170714} for $\eta = \psi(\xi)$. We clearly have \ref{D_I(X):1}, \ref{D_I(X):2}, and \ref{D_I(X):3}. In analogy to the chain rule for real derivatives, on $\H_X$ we also have
\begin{align} 
\hat D \eta (0) &= \sum_{k=1}^{2 n} \hat D_k \psi(0) \hat D  \hat \xi^{(k)}(0)
	= \hat D  \psi(0)  \hat D\hat \xi(0);  \label{eq:200521.2}\\
\hat D^2  \eta (0) &= \sum_{k,l = 1}^{2n} \hat D^2_{k,l}  \psi  \left(0\right)   \hat D  \hat \xi^{(k)} (0)^\top  \hat D  \hat \xi^{(l)} (0) + \sum_{k = 1}^{2 n} \hat D_k\psi\left(0\right) \hat D^2 \hat \xi^{(k)}(0). \label{eq:200521.3}
\end{align} 
By assumption, we have $\indicator{\H_Y} \hat D \psi(0) \in L(\hat Y)$. Since $\H_ X \subset \H_Y$, this also yields $\indicator{\H_X} \hat D \psi(0) \in L(\hat Y)$. Together with \eqref{eq:200521.1}, we obtain 
\[
\indicator{\H_X} \hat D \psi(0) \in L\left(\indicator{\H_X}  \hat D \hat  \xi (0) \cdot \hat X \right),   
\] 
hence also Definition~\ref{D:170714}\ref{D_I(X):4}  with $\xi$ replaced by $\eta = \psi(\xi)$. 
Similarly, we also get  Definition~\ref{D:170714}\ref{D_I(X):5}.

Next, observe in view of identity \eqref{eq:200521.2} that 
\begin{align*}
	\eta - \indicator{\H_X}   \hat D \eta(0) \, \hat \id   &= \left(\psi(\xi) - \indicator{\H_Y} \hat D \psi (0) \hat \xi \right)  
		+  \indicator{\H_X} \hat D  \psi (0)\left(\hat \xi - \hat D \hat \xi(0) \,\hat \id \right)  
		+  \indicator{\H_Y \setminus \H_X}  \hat D  \psi (0) \hat \xi \\&\in L_\sigma(\mu^X)
\end{align*}
by Proposition~\ref{P:190411}, the assumptions, and $\indicator{\H_Y \setminus \H_X}   | \xi| * \mu^X = 0$. This yields Definition~\ref{D:170714}\ref{D_I(X):6}  with $\xi$ replaced by $\eta$.

Finally,  \eqref{composition} follows from \eqref{eq:200521.1}, \eqref{eq:200521.2}, and \eqref{eq:200521.3}  together with the computations
\begin{align*}
		\psi \circ (\xi \circ X)
		  ={}& \indicator{\H_Y}   \hat D \psi (0) \cdot \hat Y+ \frac{1}{2} \hat D^2 \psi(0) \cdot \bigs[\hat Y,\hat Y\bigs]^c + \left(\psi - \indicator{\H_Y}   \hat D \psi(0) \, \hat \id  \right) \star \mu^{Y}\\
			={}& \indicator{\H_Y}\hat D\psi(0) \cdot  \biggl(\indicator{\H_X}\hat D \hat\xi(0) \cdot \hat X+ \frac{1}{2} \hat D^2 \hat\xi(0) \cdot \bigs[\hat X,\hat X\bigs]^c + \left(\hat\xi -\indicator{\H_X}\hat D \hat\xi(0)\hat\id  \right) \star \mu^X\biggr)\\
			+{}& \frac{1}{2} \hat D^2 \psi(0) \cdot \Biggl(\sum_{i,j  = 1}^d  \hat D_i \hat\xi(0)  \hat D_j \hat\xi(0)    \cdot \bigs[\hat X^{(i)},\hat X^{(j)}\bigs]^c\Biggr)  +  \Bigl(\psi(\xi) -  \indicator{\H_Y}\hat D \psi(0)\hat \xi \Bigr) \star \mu^X\\
			={}& \indicator{\H_X}\hat D\eta(0) \cdot \hat X + \frac{1}{2}  \hat D^2 \eta(0)     \cdot \bigs[\hat X,\hat X\bigs]^c
			+ \left(\eta - \indicator{\H_X}\hat D\eta(0)\hat\id \right) \star \mu^X\\
			={}& \psi (\xi) \circ X.
\end{align*}
Here, we have used the associativity of the stochastic integrals with respect to $\hat X$ and $[\hat X,\hat X]^c$ as well as the associativity of the $\star$ jump-measure integral.
\end{proof}

Example~\ref{Ex:3} below shows that without the assumption that $\psi$ is twice real-dif\-fer\-en\-tiable at zero, $(\P \times A^{X})$--a.e., the conclusion of Theorem~\ref{T:composition} does not necessarily hold.

\begin{remark}[The linear case] \label{R:200530}
	If $\xi \in \I(X)$ is linear, i.e., of the form $\xi= \zeta\id$ for some predictable process $\zeta$, then \eqref{eq:200521.1} is automatically satisfied. 
	\qed
\end{remark}

\begin{corollary}[Sufficient condition for composition of representations]\label{C:200522} 
	Let $\xi \in \I(X)$. Moreover, let $\psi \in \I(\xi \circ X)$ such that $\psi(0) = 0$ and   $\indicator{\H_{\xi \circ X}}  \psi$ is twice real-differentiable at zero, $(\P \times A^{X})$--a.e., and such that $\indicator{\H_{\xi \circ X}} \hat D\psi(0)$ is locally bounded (e.g., if $\psi \in \Uni$ and $\xi \circ X \in \Dom(\psi)$). Then $\psi(\xi) \in \I(X)$ and \eqref{composition} holds.
\end{corollary}

\begin{remark}[Algebra of $X$--representable processes]
Thanks to Proposition~\ref{P:1_J}\ref{P:1_J:i}, the space of $\Cx$--valued $X$--representable processes is a vector space. 
Proposition~\ref{P:Ito} and Corollary~\ref{C:200522} yield that this space is  also an algebra, namely closed under multiplication. 
Indeed, for  $U =  U_0 + \xi^U \circ X$ and $V =  V_0 + \xi^V \circ X$ we have
\[ \pushQED{\qed}
UV = U_0 V_0 
+\left( \left(U_{-}+\xi^U\right) \left(V_{-}+\xi^V\right) - U_{-}  V_{-}\right) \circ X.
\qedhere\popQED \]
\end{remark}

\begin{remark}[Converse of the composition theorem]  \label{R:200521.2}
A reverse direction of Theorem~\ref{T:composition} holds, too. To wit,  let $\xi \in \I^n(X)$. Moreover, fix some predictable function $\psi:  \overline \Omega^n \rightarrow  \Cinf$
 such that $\psi(\xi) \in \I^1(X)$, $\indicator{\H_X} \psi$ is twice real-differentiable at zero, $(\P \times A^{X})$--a.e., and 
\eqref{eq:200521.1}	holds. Then $\psi \in \I(\xi\circ X)$.  

To see this, first note that $\psi(0) = \psi(\xi(0)) = 0$,  $(\P \times A^{X})$--a.e.  We next follow the arguments of Theorem~\ref{T:composition}, using \eqref{eq:200521.2} and \eqref {eq:200521.3} with $\eta = \psi(\xi) \in \I(X)$. We then directly obtain that Definition~\ref{D:170714}\ref{D_I(X):1}, \ref{D_I(X):2}, \ref{D_I(X):3}, \ref{D_I(X):4}, and \ref{D_I(X):5} hold with $\xi$ replaced by $\psi$. Here we used again $\xi(\Delta X) = 0$ on $\H_ X \setminus \H_{\xi\circ X}$.  Next, observe that
 \begin{align*}
	\psi(\xi) - \indicator{\H_{\xi\circ X}} \hat D  \psi (0) \hat \xi  &=  \left(\eta -  \indicator{\H_X} \hat D \eta (0) \,\hat \id \right) +  \indicator{\H_X}  D \hat \psi (0)   \left(\hat D \hat \xi (0) \, \hat \id  -  \hat \xi\right) \in  L_\sigma(\mu^X),
\end{align*}
yielding the claim.

Examples~\ref{Ex:1a} and \ref{Ex:1a''} illustrate again how essential \eqref{eq:200521.1} is for the remark to hold.
\qed
\end{remark}

\begin{corollary}[Inverse functions]  \label{C:170726}
Let $\xi \in \I^d(X)$ and $Y = Y_0 + \xi \circ X$. Moreover,  assume that the smallest singular value of $\hat D\hat{\xi}(0)$ is locally bounded away from zero and that $\xi$ allows for a predictable left inverse $\psi$ (see Subsection~\ref{SS:predictable}). Then $\psi \in \I^d(Y)$ and $X = X_0 + \psi \circ Y$. 
\end{corollary}
\begin{proof}
	Since $\hat \xi\bigs(\hat \id^{-1}\bigs)$ is continuously differentiable at zero on $\H_X$, $\hat \psi\bigs(\hat \id^{-1}\bigs)$ is actually an inverse of $\hat \xi\bigs(\hat \id^{-1}\bigs)$ in a neighbourhood of zero on $\H_X$. Thus $\indicator{\H_X} \psi$ is twice real-differentiable at zero with $\hat D \hat \psi(0) = (\hat D \hat{\xi}(0))^{-1}$ on $\H_X$.
	If now the smallest singular value of $\hat D\hat {\xi}(0)$ is locally bounded away from zero, then the largest singular value of $\hat D\hat{\psi}(0)$ is locally bounded and by equivalence of the Schatten and maximum matrix norms each element of $\hat D {\psi}(0)$ is locally bounded. The assertion follows from Remark~\ref{R:200521.2}. 
\end{proof}
If the assumption that the smallest singular value of $\hat D\hat {\xi}(0)$ is locally bounded away from zero is replaced by the weaker assumption that is is merely positive, then Corollary~\ref{C:170726} is wrong as Examples~\ref{Ex:1a} and \ref{Ex:1a''} below illustrate, even if $\psi$ is an inverse of $\xi$ and  $d=1$.

\begin{remark}[Advantages of the proposed calculus]
	Results like Theorem~\ref{T:composition}  make this stochastic calculus simple and powerful. 	Consider the situation when one has to perform a change of variables $f(Y)$ on an $X$--representable $\R$--valued process $Y$.  A direct application of the It\^o--Meyer formula \eqref{eq:170702} to the representation of $Y$ in \eqref{eq:170704} yields
\begin{equation*}
	\begin{split}
f(Y) = f(Y_0)  + Df&(Y_-) \cdot \left(D \xi (0) \cdot X+ \frac{1}{2} D^2 \xi(0) \cdot [X,X]^c 
+ \left(\xi - D \xi(0) \, \id \right) \star \mu^X\right) \\
&+ \frac{1}{2} D^2 f(Y_-) \cdot [Y,Y]^c + (f(Y_- + \id) - f(Y_-) - Df(Y_-) \,\id) * \mu^Y.
	\end{split}
\end{equation*}
One then has to collect all terms manually in order to simplify this expression and eventually recast it in terms of $\mu^X$.  
	
In contrast, the notation of \eqref{eq:170704.4} gives $f(Y) = f(Y_0) + \xi^{f, Y}(\xi) \circ X$.
Only the function $\xi^{f, Y} (\xi)$  needs to be computed and then the corresponding representation applies.
This is pedagogically pleasing because $\xi^{f, Y} (\xi)$ describes the jumps of $f(Y)$ in terms of the jumps of $X$, i.e.,
\[\pushQED{\qed} \Delta f(Y) = f(Y_- +\Delta Y)-f(Y_-) = f(Y_- +\xi(\Delta X))-f(Y_-) = \xi^{f, Y}(\xi(\Delta X)).
\qedhere\popQED \] 
\end{remark}

\subsection{Alternative  \'Emery formula} \label{SS:Wirtinger}
In the non-analytic case, which too is of practical importance, it can be helpful to rephrase the \'Emery formula~\eqref{eq:190610.2} in terms of the $\Cx^{2d}$--valued process $(X,X^*)$. Here $X^*$ denotes the complex conjugate of $X$.  This allows the use of Wirtinger partial derivatives (see \cite{wirtinger.27}), given by
\begin{equation}\label{eq:Wirtinger}
\frac{\partial}{\partial x} = \frac{1}{2}\bigg(\frac{\partial}{\partial \Re x}-i \frac{\partial}{\partial \Im x}\bigg)\quad\text{and}\quad \frac{\partial}{\partial x^{*}} = \frac{1}{2}\bigg(\frac{\partial}{\partial \Re x}+i \frac{\partial}{\partial \Im x}\bigg).
\end{equation}
This turns out to be convenient in some applications; see Proposition~\ref{P:190114.1} and Example~\ref{E:190302}. Observe, however, that 
the proposed calculus allows one to write simply $\xi\circ X$ and operate on the level of $\xi$, where the specific physical implementation of $\xi\circ X$ is immaterial.

To arrive at the alternative \'Emery formula, we introduce the function  $\check \id:\Cinf^d\to\Cinf^{2d}$ by
\begin{equation*}
\check{\id} = 
\left(I_d\otimes \left[\begin{array}{cc}1 & i\\1 &-i\end{array}\right]\right) \hat{\id}; \qquad \check \id(\NaN) = \NaN,
\end{equation*}
where $\otimes$ again denotes  the Kronecker product.
This allows us to introduce the process
\begin{align}\label{eq:190704.2}
\check{X}=\check{\id}(X).
\end{align} 
Observe that $\hat{X}$ is the $\R^{2d}$--valued process containing the values of $\Re X$ and $\Im X$, interlaced, while $\check{X}$ is the $\Cx^{2d}$--valued process containing $X$ and its conjugate $X^*$, interlaced.

Next, we denote by $\check{D}\xi$ the row vector of Wirtinger derivatives, given by
\begin{equation}\label{eq:190115.3}
\check{D}\xi = \frac{1}{2} \hat D{\xi} \left( I_d\otimes \left[\begin{array}{cc}1 & 1\\-i &i\end{array}\right]\right),
\end{equation} 
and by $\check{D}^2\xi$ the corresponding `Wirtinger Hessian,' given by
\begin{equation*}
\check{D}^2 \xi^{(k)} =\check{D} (\check D \xi^{(k)})^\top =  \left( I_d\otimes \left[\begin{array}{cc}1 & -i\\1 &i\end{array}\right]\right)  \frac{1}{4} \hat D^2{\xi}^{(k)} \left( I_d\otimes \left[\begin{array}{cc}1 & 1\\-i &i\end{array}\right]\right), \qquad k = 1, \ldots, n.
\end{equation*} 

The following technical observation will be very useful in the subsequent proposition.
\begin{lemma}[Invertible linear transformations in a stochastic integral]\label{L:200530}
Fix $m \in \N$. 
Let $\Lambda_1$, $\Lambda_2$ be arbitrary invertible matrices in $\Cx^{m \times m}$.  Let $\zeta$ denote a $\Cx^{m \times m}$--valued predictable process and let $V$ denote a  $\Cx^{m \times m}$--valued semimartingale.
Then the following are equivalent.
\begin{enumerate}[label={\rm(\roman{*})}, ref={\rm(\roman{*})}]
\item\label{phiG1} $\zeta \in L(V)$.
\item\label{phiG2} $\Lambda_1\zeta \Lambda_2 \in L\left(\Lambda_2^{-1}V\Lambda_1^{-1}\right)$.
\end{enumerate}
If one (hence both) of these conditions holds, then 
\begin{align} \label{eq:200530}
\zeta \cdot V = \Lambda_1\zeta \Lambda_2 \cdot \left(\Lambda_2^{-1}V\Lambda_1^{-1}\right).
\end{align}
\end{lemma}
\begin{proof}
Note that it suffices to argue the implication from \ref{phiG1} to \ref{phiG2} and to show \eqref{eq:200530}. Moreover, since $\zeta \in L(V)\iff \zeta^\top \in L(V^\top)$ and $\zeta \cdot V = \zeta^\top \cdot V^\top$, it is enough to prove the statement with $\Lambda_1$ being the identity matrix.  To this end, assume \ref{phiG1} holds. Let $\vecr(\zeta)$ (respectively, $\vecc(V)$) denote the row-wise (column-wise) flattening of $\zeta$ (respectively, $V$), that produces a $(1 \times m^2)$--dimensional row ($m^2$--dimensional column) vector. Then \ref{phiG1} is equivalent to $\vecr(\zeta) \in L(\vecc(V))$ and one has $\zeta \cdot V = \vecr(\zeta) \cdot \vecc(V)$. 
Thanks to Proposition~\ref{P:integral} and Remark~\ref{R:200521.2}, this then yields $(\vecr(\zeta) R)  \in L(R^{-1}\vecc(V))$ for any invertible $m^2\times m^2$ matrix $R$, along with 
$$\zeta \cdot V = (\vecr(\zeta) R) \cdot (R^{-1} \vecc(V)).$$
 Choosing $R = I_m\otimes \Lambda_2$ yields $R^{-1} = I_m\otimes \Lambda^{-1}_2$, $\vecr(\zeta) R=\vecr(\zeta \Lambda_2)$, 
and hence the desired statement.
\end{proof}

\begin{proposition}[\'Emery formula in terms of Wirtinger derivatives] \label{P:190114.1}
For $\xi \in \I(X)$, the following terms are well defined and we have
\begin{align}  
	\xi\circ X 	 &=	\indicator{\H_X}   \check D \xi (0) \cdot \check X+ \frac{1}{2}\check D^2  \xi(0) \cdot \left[\check X,\check X\right]^c + \left(\xi - \indicator{\H_X}   \check D \xi(0) \, \check \id  \right) \star \mu^{X}.
\label{eq:190114.3}
\end{align}	
Furthermore, if $\indicator{\H_X}  \xi$ is analytic at $0$, $(\P \times A^X)$--a.e., the following terms are well defined and we have
\begin{equation} \label{eq:190703.1}
	\xi\circ X  =	\indicator{\H_X}   D \xi (0) \cdot  X+ \frac{1}{2}  D^2  \xi(0) \cdot \left[ X,X\right]^c + \left(\xi - \indicator{\H_X}   D \xi(0) \,  \id  \right) \star \mu^{X}.
\end{equation}
Here $D\xi(0)$ and $D^2\xi(0)$ stand for complex derivatives. 
\qed
\end{proposition} 

\begin{proof}
Let us first prove \eqref{eq:190114.3}. To this end, we introduce the matrix 
$$\Sigma =\frac{1}{2} I_d \otimes \left[\begin{array}{cc}1 & 1\\-i &i\end{array}\right],$$
satisfying $\Sigma^{-1} = 2(\Sigma^*)^\top $. 
Then $\check D \xi = \hat D \xi \Sigma$ and $\check X = \Sigma^{-1} \hat X$. Thanks to Proposition~\ref{P:integral} and Remark~\ref{R:200521.2}, we then have $\indicator{\H_X}   \check D \xi (0) \in L(\check X)$ and 
$\indicator{\H_X}  \check D \xi (0) \cdot \check X = \indicator{\H_X}   \hat D \xi (0) \cdot \hat X$. By the same token, we also have $\check D \xi(0) \, \check \id = \hat D \xi(0) \, \hat \id$.
Next, note that
\begin{align*}
\check D^2  \xi(0) = \Sigma^{\top} \hat D^2  \xi(0) \Sigma; \qquad \bigs[\check{X},\check{X}\bigs] = \Sigma^{-1} \bigs[\hat{X},\hat{X}\bigs] (\Sigma^{-1})^\top.
\end{align*}
An application of Lemma~\ref{L:200530}  now concludes the proof of  \eqref{eq:190114.3}.

The simplifications in the analytic case follow from the standard properties of Wirtinger derivatives, see for example Remmert~\cite[Section~I.4]{remmert.91}. 
\end{proof}

We now provide two examples of complex-valued representations where the representing functions are not assumed analytic at 0. Recall the notation for Wirtinger derivatives in \eqref{eq:Wirtinger}.

\begin{example}[Quadratic covariation of represented semimartingales]
Let $X$ be a $\Cx$--valued semimartingale. Then by Proposition~\ref{P:1}\ref{P:1:iii} and Corollary~\ref{C:200522}, for  $\xi, \psi \in \I^1(X)$, we have 
$[\xi \circ X, \psi \circ X] = \xi \psi\circ X$. 
In the explicit form \eqref{eq:190114.3}, this is written as
\begin{align*}
	[\xi \circ X, \psi \circ X] ={}& \partial_x \xi(0) \partial_x\psi(0) \cdot [X,X]^c 
	+ \partial_{x^{*}} \xi(0) \partial_{x^{*}}\psi(0)\cdot [X^*,X^*]^c\\
	& +(\partial_x \xi(0) \partial_{x^{*}}\psi(0)+ \partial_{x^{*}}\xi(0)\partial_{x}\psi(0))\cdot [X,X^*]^c + \xi \psi \star\mu^X.
\end{align*}
This formula seems very intuitive. The first three terms capture the continuous covariation of $\xi \circ X$ and $\psi \circ X$. The last term is the pure-jump component which multiplies together the jumps in $\xi \circ X$ and $\psi \circ X$.
\qed
\end{example}

\begin{example}[Explicit complex-valued expression for $(\abs{1+\id}^\alpha-1)\circ X$, $\alpha\in\Cx$, $\Delta X\neq -1$]\label{E:190302}
Consider the predictable function $\xi=\abs{1+\id}^\alpha-1$, which on a sufficiently small neighbourhood of zero satisfies
$$\abs{1+\id}^\alpha-1 =  (1+\id)^{\frac{\alpha}{2}}(1+\id^*)^{\frac{\alpha}{2}}-1.$$ 
On this neighbourhood, apply formal Wirtinger calculus to obtain
\begin{align*}
\partial_{x} \xi &{}= \frac{\alpha}{2}(1+\id)^{\frac{\alpha}{2}-1}(1+\id^*)^{\frac{\alpha}{2}};\qquad \qquad \quad \ \ \,
\partial_{x^*} \xi  = \frac{\alpha}{2}(1+\id)^{\frac{\alpha}{2}}(1+\id^*)^{\frac{\alpha}{2}-1}; \\
\partial_{xx}^2 \xi &{}= \frac{\alpha}{2}\left(\frac{\alpha}{2}-1\right)(1+\id)^{\frac{\alpha}{2}-2}(1+\id^*)^{\frac{\alpha}{2}};\quad
\partial_{x^*\!x^*}^2 \xi = \frac{\alpha}{2}\left(\frac{\alpha}{2}-1\right)(1+\id)^{\frac{\alpha}{2}}(1+\id^*)^{\frac{\alpha}{2}-2};\\
& \qquad\qquad\qquad\qquad\qquad\quad   \partial_{xx^*}^2 \xi = \frac{\alpha^2}{4}\abs{1+\id}^{\alpha-1}.
\end{align*}
Next, $\xi \in \Uni$, hence $\xi \in \I(X)$ for \emph{any} $\Cx$--valued semimartingale $X$ with $X\in \Dom(\xi)$,  in particular for any $X$ with $\Delta X\neq -1$. Formula \eqref{eq:190114.3} now yields
\begin{align*} 
(\abs{1+\id}^\alpha-1)\circ X ={}& \frac{\alpha}{2} \cdot (X+X^*)
+\frac{\alpha}{4}\left(\frac{\alpha}{2}-1\right)([X,X]^c+[X^*,X^*]^c)+\frac{\alpha^2}{4}[X,X^*]\\
&\quad + \left(\abs{1+\id}^\alpha-1-\frac{\alpha}{2}(\id+\id^*)\right) * \mu^X\\
={}&\  \alpha\cdot \Re X
+\frac{\alpha}{2}(\alpha-1)[\Re X,\Re X]^c+\frac{\alpha}{2}[\Im X,\Im X]^c\\
&\quad + \left(\abs{1+\id}^\alpha-1-\alpha\Re \id)\right) * \mu^X.
\end{align*}
We continue discussing this setup in Example~\ref{E:200609} below and apply it in \cite[Examples~4.3 and 4.4]{crIII}. There one obtains the Mellin transform of the positive and negative parts of a signed stochastic exponential of a process with independent increments.
\qed
\end{example}

\setlength{\abovedisplayskip}{5pt}
\setlength{\belowdisplayskip}{5pt}

\section{Specific examples of the semimartingale representation} \label{S:4}
 
 \subsection{Generic applications}
 If $X$ is a $\Cx$--valued semimartingale, then by \cite[Th\'eor\`eme~1]{doleans-dade.70} (see also \cite[I.4.60]{js.03}) the stochastic exponential $\Exp(X)$ of $X$ is the unique solution to the stochastic differential equation 
\begin{equation}\label{eq:190701.9}
 \Exp(X)=1+\Exp(X)_{-}\cdot X.
\end{equation}
 The stochastic logarithm $\Log(X)$ of a semimartingale $X$ that can hit zero only by a jump (but not continuously) and is absorbed in zero is given by
\begin{align*}
	\Log(X) = \frac{1}{X_-}\indicator{\{X_-\neq 0\}} \cdot X,
\end{align*}
where 
 $\sfrac{\indicator{\{X_{t-}\neq 0\}}}{X_{t-}}$ is defined to be zero on the set  $\{X_{t-}=0\}$, for all $t\geq 0$.

All representing functions shown in this subsection belong to the universal class $\Uni$  
and can therefore be applied to any semimartingale whose jumps are compatible with the given function (Proposition~\ref{P: universal}).
The simplified stochastic calculus yields many identities by straightforward computations. Using only the It\^o--Meyer change of variables formula, those identities would involve convoluted arguments. As an example, we now establish a generalization of Yor's formula and its converse (see \cite[II.8.19--20]{js.03}).
\begin{proposition}[Generalized Yor formula and its converse]\label{P:190701b}{}
Consider a $\Cx^2$--valued semimartingale $X$ and  $\alpha$, $\beta\in\Cx$. 
\begin{enumerate}[label={\rm(\arabic{*})}, ref={\rm(\arabic{*})}]
\item\label{P:190701b.1}  Assume that the following conditions hold.
\begin{itemize}
	\item If $\alpha \in \{0, -1, -2, \ldots\}$, then $\Delta X^{(1)} \neq -1$. 
	\item If $\alpha \in \Cx \setminus \Z$, then $\Re \Exp\bigs(X^{(1)}\bigs)>0$. 
\end{itemize}
Assume that these two conditions also hold with $\alpha$ and $X^{(1)}$ replaced by $\beta$ and $X^{(2)}$, respectively.
We then have
\begin{equation}\label{eq:190701.5}
\Exp \bigs(X^{(1)}\bigs)^\alpha\Exp \bigs(X^{(2)}\bigs)^\beta = \Exp \left( \left((1+\id_1)^\alpha (1+\id_2)^\beta -1\right)\circ X \right),
\end{equation}
where complex powers with exponent in $\Cx\setminus\Z$ are defined via the principal value logarithm.
In particular, with $\alpha = \beta = 1$ we have
\begin{align} \label{eq:200528}
\Exp \bigs(X^{(1)}\bigs) \Exp \bigs(X^{(2)}\bigs) = \Exp \left(X^{(1)} + X^{(2)} + \bigs[X^{(1)},X^{(2)}\bigs]\right).
\end{align}
\item\label{P:190701b.2}
Assume next the following conditions.
\begin{itemize}
	\item If $\alpha \in \N $, then $X^{(1)}$ does not reach zero continuously and is absorbed in zero.
	\item If $\alpha \in \{0, -1, -2, \ldots\} $, then $X^{(1)} \neq 0$ and $X^{(1)}_- \neq 0$. 
	\item If $\alpha \in \Cx \setminus \Z$, then $\Re X^{(1)} > 0$ and $\Re X^{(1)}_- > 0$. 
\end{itemize}
Assume that these three conditions also hold with $\alpha$ and $X^{(1)}$ replaced by $\beta$ and $X^{(2)}$, respectively. Finally, denote by $\tau$ the first time $X^{(1)}$ or $X^{(2)}$ hit zero.
We then have
\begin{equation} \label{eq:flight_to_BKK}
\Log \bigs(\bigs(X^{(1)}\bigs)^\alpha \bigs(X^{(2)}\bigs)^\beta\bigs) 
=   \left((1+\id_1)^\alpha (1+\id_2)^\beta -1\right)\circ \bigs(\Log\bigs(X^{(1)}\bigs)^\tau,\Log\bigs(X^{(2)}\bigs)^\tau\bigs).
\end{equation}
In particular, with $\alpha = \beta = 1$ and $X^{(1)},X^{(2)}$ not hitting zero we have 
\[
\Log \bigs(X^{(1)}  X^{(2)}\bigs) = \Log \bigs(X^{(1)}\bigs) + \Log \bigs(X^{(2)}\bigs) + \bigs[\Log \bigs(X^{(1)}\bigs), \Log \bigs(X^{(2)}\bigs)\bigs].
\]
\end{enumerate}
\end{proposition}
\begin{proof}
To start, 
from \eqref{eq:190701.9} for example, recall that by Proposition~\ref{P:integral}
\begin{align}\label{eq:200604.4}
	\Exp\bigs(X^{(k)}\bigs) = 1 + \Exp\bigs(X^{(k)}\bigs)_-\id_k \circ X,\qquad k\in\{1,2\}.
\end{align}
Now, the change of variables formula in Proposition~\ref{P:Ito} applied to the function 
$f= \id_{1}^{\alpha} \id_{2}^\beta$ over an appropriate domain $\mathcal{U}$ (obtained as the Cartesian product of appropriate one-dimensional domains, i.e., $\Cx$ for $\alpha\in\N$; 
$\Cx\setminus\{0\}$ for $\alpha\in \{0,-1,\ldots\}$; $\{z \in \Cx: \Re z > 0\}$ for $\alpha\in\Cx\setminus \Z$; and likewise with $\beta$ in place of $\alpha$) 
combined with Proposition~\ref{P:integral} and Theorem~\ref{T:composition0} yield
\begin{align*}
\Exp\bigs(X^{(1)}\bigs)^\alpha\Exp\bigs(X^{(2)}\bigs)^\beta
	&= 1 + \left(\Exp\bigs(X^{(1)}\bigs)_{-}^\alpha\Exp\bigs(X^{(2)}\bigs)_{-}^\beta\right)\cdot
	\left( ( 1+\id_{1})^\alpha( 1+\id_{2})^\beta-1\right)\circ X.
\end{align*}
The uniqueness of strong solutions to the stochastic differential equation~\eqref{eq:190701.9} then yields \eqref{eq:190701.5}.

Next, define $Y^{(k)} = \Log(X^{(k)})$ for $k \in \{1,2\}$. Then from \eqref{eq:190701.5} we obtain
\[
\bigs(X^{(1)}\bigs)^\alpha \bigs(X^{(2)}\bigs)^\beta = \Exp \bigs(Y^{(1)}\bigs)^\alpha\Exp \bigs(Y^{(2)}\bigs)^\beta 
= \Exp \bigs( \bigs((1+\id_1)^\alpha (1+\id_2)^\beta -1\bigs)\circ Y\bigs).
\]
Taking stochastic logarithms on both sides yields \eqref{eq:flight_to_BKK}.
\end{proof}
With $\alpha = 1$ and $\beta = -1$, identity \eqref{eq:190701.5} yields a $\Cx$--valued extension of Equation~(1-5) in M\'emin~\cite{memin.78}. If we additionally assume $X^{(1)} = 0$ and $Y = X^{(2)}$,  we get 
\[
	\Log\left(\frac{1}{\Exp (Y)}\right) = \left(\frac{1}{1 + \id} - 1 \right) \circ  Y = Y_0 - Y + [Y, Y]^c + \frac{\id^2}{1+\id} * \mu^{Y}.
\]
Here we used that the function $(1+\id)^{-1}-1\in\Uni$ is analytic at $0$ in conjunction with Propositions~\ref{P: universal} and 
\ref{P:190114.1}.   See also Larsson and Ruf \cite[Theorem~4.1]{larsson.ruf.19} for an $\R$--valued version.

Parts of the following proposition, restricted to real-valued semimartingales, appear in  Jacod and Shiryaev~\cite[II.8.8--12]{js.03}.
\begin{proposition}[Identities involving natural\,/\,stochastic exponential\,/\,logarithm]\label{P:190701}
Let $X$ denote a  $\Cx$--valued semimartingale.  Then
\begin{align} 
\Log \bigs(\e^{X}\bigs)&=(\e^{\id}-1)\circ X;\label{eq:190701.1}\\
\lvert\Exp (X)\rvert&= \Exp\left(\left(\lvert 1+\id\rvert - 1\right) \circ X\right).\label{eq:190228.1}
\end{align}
If $\Delta X\neq -1$, then 
\begin{align} 
\Exp (X) &= \e^{\log (1+\id) \circ X};  \label{eq:190128.2}\\
\log\lvert\Exp (X)\rvert &= \log \lvert 1+\id\rvert \circ X, \nonumber
\end{align}
where $\log$ denotes again the principal value logarithm. 
 Moreover, if $\Re \Exp(X)>0$ then  
\begin{equation*}
\log \Exp(X) = \log (1+\id)\circ X.
\end{equation*}

\end{proposition}

\begin{proof}  Apply the change of variables formula in Proposition~\ref{P:Ito} to the function $\e^{\id}$ to obtain 
\begin{equation*}
\e^{X}= \e^{X_0}+\bigs(\e^{X_{-}+\id}-\e^{X_{-}}\bigs)\circ X.
\end{equation*}
As $\sfrac{1}{\e^{X_{-}}}$ is locally bounded, Proposition~\ref{P:integral} in conjunction with Theorem~\ref{T:composition0} yield
\begin{equation*}
\Log \bigs(\e^{X}\bigs) =\e^{-X_{-}} \cdot \e^X= \e^{-X_{-}} \bigs(\e^{X_{-}+\id}-\e^{X_{-}}\bigs)\circ X,
\end{equation*}
which on simplification gives \eqref{eq:190701.1}.

Next we will prove~\eqref{eq:190228.1} under the additional assumption $\Delta X\neq -1$. Since the function $\abs{\id}$ is twice continuously real-differentiable on $\mathcal U = \Cx \setminus \{0\}$ and $\Exp(X)_-,\Exp(X)$ take values in $\mathcal{U}$, Propositions~\ref{P:integral} and \ref{P:Ito}, representation~\eqref{eq:200604.4}, and Theorem~\ref{T:composition0} yield
\begin{align*}
	\bigs\lvert\Exp(X)\bigs\rvert
	={}& \left(\left\lvert \Exp(X)_- + \id\right\rvert - \left\lvert\Exp(X)_-\right\rvert\right) \circ \Exp(X) 
	=\left(\left\lvert \Exp(X)_- + \Exp(X)_-\id\right\rvert - \left\lvert\Exp(X)_-\right\rvert\right) \circ X\\
	={}&\left(\left\lvert \Exp(X)_-\right\rvert (\left\lvert1 + \id\right\rvert - 1)  \right) \circ X
	=\left\lvert \Exp(X)_-\right\rvert\cdot \left( (\left\lvert1 + \id\right\rvert - 1)   \circ X\right),
\end{align*}
therefore \eqref{eq:190228.1} holds in this special case.

Define next
\begin{equation}\label{eq:190306.1}
Y^{(1)}=\id\indicator{\id\neq-1}\circ X\qquad\text{and}\qquad Y^{(2)}=-\indicator{\id= -1}\circ X.
\end{equation} 
Observe that $X=Y^{(1)}+Y^{(2)}$, 
$[Y^{(1)},Y^{(2)}]=0$, and $\lvert\Exp(Y^{(2)})\rvert=\Exp(Y^{(2)})$ 
as the latter only takes on values 0 and 1. The Yor formula in \eqref{eq:200528} now yields
\begin{equation}\label{eq:200604.3}
\abs{\Exp(X)}=\mathopen{\bigs|}\Exp\bigs(Y^{(1)}\bigs)\Exp\bigs(Y^{(2)}\bigs)\mathclose{\bigs|}
=\mathopen{\bigs|}\Exp\bigs(Y^{(1)}\bigs)\mathclose{\bigs|}\Exp\bigs(Y^{(2)}\bigs).
	\end{equation}
Moreover,  note that $\Delta Y^{(1)}\neq-1$ hence by the special case of \eqref{eq:190228.1} shown earlier we have
\begin{equation}\label{eq:190228.2}
\lvert\Exp(Y^{(1)})\rvert
= \Exp\left(\left(\lvert 1+\id\rvert - 1\right) \circ Y^{(1)}\right)
=\Exp\bigs((\abs{1+\id\indicator{\id\neq-1}}-1)\circ X\bigs),
\end{equation}
where the second equality follows from \eqref{eq:190306.1} and Theorem~\ref{T:composition0}. 
 Equations \eqref{eq:190306.1}--\eqref{eq:190228.2} now yield
$$
\abs{\Exp(X)} 
	=\Exp\bigs((\abs{1+\id\indicator{\id\neq-1}}-1)\circ X\bigs)\Exp\left(-\indicator{\id= -1}\circ X\right)
$$
 and a second application of the Yor formula \eqref{eq:200528} concludes the proof of \eqref{eq:190228.1} in full generality.

Assume from now on that $\Delta X \neq -1$. 
Observe that $\log(1+\id)$ is the right-inverse of the function $\e^{\id}-1$ over the domain $\Cx \setminus \{-1\} $ and that $\log(1+\id) \in \Uni \cap \I(X)$. We may therefore define $Y=\log(1+\id) \circ X$. 
From \eqref{eq:190701.1} and the composition Theorem~\ref{T:composition0} one obtains
$$\Log\left(\e^{\log(1+\id) \circ X}\right) =  \Log \bigs(\e^Y\bigs)=\bigs(\e^{\id}-1\bigs)\circ Y 
= \bigs(\e^{\log(1+\id)}-1\bigs)\circ X = X-X_0,$$
which yields \eqref{eq:190128.2}.

Finally, for a semimartingale $Y$ satisfying $Y>0$ and $Y_->0$ one obtains by Proposition~\ref{P:Ito}, the identity $Y = Y_-\id\circ \Log(Y)$, and Theorem~\ref{T:composition0} that
\begin{equation*}\label{eq:200606.2}
\begin{split}
	\log Y ={}& \left(\log(Y_-+\id)-\log(Y_-)\right)\circ Y\\
 ={}& \left(\log(Y_-+Y_-\id)-\log(Y_-)\right)\circ \Log(Y) = \log(1+\id)\circ \Log(Y),
\end{split}
\end{equation*}
hence
\[
	\log\lvert\Exp (X)\rvert = \log\Exp\left(\left(\lvert 1+\id\rvert - 1\right) \circ X\right)
		= \log(1 + \id) \circ \left(\left(\lvert 1+\id\rvert - 1\right) \circ X\right) = \log\abs{1+\id} \circ X,
\]
again by composition.

Consider now $X$ such that $\Re \Exp(X)>0$, hence $\Re \Exp(X)_-\geq 0$ and $\Exp(X)_-\neq 0$. As in the previous step, by Proposition~\ref{P:Ito} and Theorem~\ref{T:composition0} one obtains
\begin{align*} 
\log \Exp(X) ={}& (\log (\Exp(X)_-+\id) - \log \Exp(X)_-)\circ \Exp(X) \\
={}&(\log (\Exp(X)_-(1+\id)) - \log \Exp(X)_-)\circ X= \log (1 + \id)\circ X,
\end{align*}
where the last equality follows by comparing the respective \'Emery formulae.
\end{proof}

\begin{proposition}[Generalized Yor formula involving absolute values]\label{P:200604}{}
Consider a $\Cx^2$--valued semimartingale $X$ and  $\alpha$, $\beta\in\Cx$. 

\begin{enumerate}[label={\rm(\arabic{*})}, ref={\rm(\arabic{*})}]
\item\label{P:200604.1}  Assume that the following condition holds.
\begin{itemize}
	\item If $\alpha \in \Cx \setminus (0,\infty)$ then $\Delta X^{(1)} \neq -1$.
\end{itemize}
Assume that this condition also holds with $\alpha$ and $X^{(1)}$ replaced by $\beta$ and $X^{(2)}$, respectively.
We then have
\begin{equation}\label{eq:200604.1}
\mathopen{\bigs|}\Exp \bigs(X^{(1)}\bigs)\mathclose{\bigs|}^\alpha\mathopen{\bigs|}\Exp \bigs(X^{(2)}\bigs)\mathclose{\bigs|}^\beta 
= \Exp \left( \left(\abs{1+\id_1}^\alpha \abs{1+\id_2}^\beta -1\right)\circ X \right).
\end{equation}
\item\label{P:200604.2}
Assume next the following conditions. 
\begin{itemize}
\item If $\alpha \in (0,\infty) $ then $X^{(1)}$ does not reach zero continuously and is absorbed in zero.
\item If $\alpha \in \Cx \setminus (0,\infty)$ then $X^{(1)} \neq 0$ and $X^{(1)}_- \neq 0$.
\end{itemize}
Assume that these two conditions also hold with $\alpha$ and $X^{(1)}$ replaced by $\beta$ and $X^{(2)}$, respectively. Finally, denote by $\tau$ the first time $X^{(1)}$ or $X^{(2)}$ hit zero. We then have
\begin{equation} \label{eq:200604.2}
\Log \left(\mathopen{\bigs|}X^{(1)}\mathclose{\bigs|}^\alpha \mathopen{\bigs|}X^{(2)}\mathclose{\bigs|}^\beta\right) 
=   \left(\abs{1+\id_1}^\alpha \abs{1+\id_2}^\beta -1\right)\circ \bigs(\Log\bigs(X^{(1)}\bigs)^\tau,\Log\bigs(X^{(2)}\bigs)^\tau\bigs).
\end{equation}
\end{enumerate}
\end{proposition}
\begin{proof}
First, for $\alpha\in\Cx\setminus (0,\infty)$ and $k\in\{1,2\}$ we have thanks to \eqref{eq:190228.1}
\begin{equation}\label{eq:210329.1}
\mathopen{\bigs|}\Exp \bigs(X^{(k)}\bigs)\mathclose{\bigs|} 
= \Exp \left( \left(\abs{1+\id} -1\right)\circ X^{(k)} \right) > 0.
\end{equation}
Formula \eqref{eq:200604.1} now follows via an application of Proposition~\ref{P:190701b}\ref{P:190701b.1} in conjunction with Theorem~\ref{T:composition0}. In the case $\alpha\in (0,\infty)$, the function $\abs{\id}^\alpha$ is well defined on $\Cx$ rather than just $\Cx\setminus\{0\}$, hence condition \eqref{eq:210329.1} is not needed. For $\alpha\in\N$ we may appeal again to Proposition~\ref{P:190701b}\ref{P:190701b.1} to obtain \eqref{eq:200604.1} and the result easily carries over to all $\alpha \in (0,\infty)$. Item \ref{P:200604.1} is proved.

Representation \eqref{eq:200604.2} is now obtained by writing $X^{(k)}=\Exp\bigs(\Log\bigs(X^{(k)}\bigs)\bigs)$ for each $k\in\{1,2\}$, applying formula \eqref{eq:200604.1} to the stochastic exponentials, and finally taking stochastic logarithms on both sides of \eqref{eq:200604.1}. 
\end{proof}

\begin{example}[Iterated composition] \label{Ex:170812}
Let us now consider  the following
construction for a $\Cx$--valued semimartingale $X$ and for a constant $\alpha \in \Cx$. Define inductively the processes $Y^0=X$;
\begin{align*}
Y^k &=\Log\bigs( \exp\bigs(\alpha Y^{k-1}\bigs)\bigs), \qquad k \in \N.
\end{align*}
Then an induction argument, \eqref{eq:190701.1}, and Theorem~\ref{T:composition0} yield that $Y^k = \xi^k \circ X$ for all $k \in \N \cup \{0\}$, with $\xi^0 = \id$;
\begin{align*}
\xi^k &= \exp \bigs( \alpha \xi ^{k-1}\bigs) -1, \qquad k \in \N.
\end{align*}
Explicitly, for each $k \in \N$, $\xi ^k$ is a nested function of the form 
\begin{equation*}
\xi^k=\underbrace{\exp \left( \alpha \left( \ldots \left( \exp
\left( \alpha \left( \exp \left( \alpha\, \id\right) -1\right) \right) -1\right)
\ldots \right) \right) -1}_{k \text{ times}}.
\end{equation*}

Using the chain rule, one infers that $\xi^k$ is analytic at zero for each $k \in \N$  with
\begin{align*}
D \xi^k (0) &= \alpha D \xi^{k-1}(0);\\
D^2 \xi^k (0) &= \alpha D^2\xi^{k-1}(0) +\alpha^{2} \bigs(D \xi^{k-1}(0)\bigs)^{2},
\end{align*}
which implies
\begin{align*}
D\xi ^k(0) &=\alpha ^{k}; \qquad
D^2\xi ^k(0) =\alpha D^2\xi^{k-1}(0) + \alpha ^{2k} =\alpha ^{k+1}\frac{\alpha ^{k}-1}{\alpha -1},\qquad\qquad k \in \N,
\end{align*}
where for $\alpha=1$ we interpret $\sfrac{(\alpha ^{k}-1)}{(\alpha -1)}$ as $k$. We conclude that, for each $k \in \N$,
\begin{equation}
Y^k=\alpha ^{k}(X - X_0) +\frac{1}{2}\alpha ^{k+1}\frac{\alpha ^{k}-1}{\alpha-1} [X,X]^c 
+\bigs( \xi ^{k} - \alpha ^{k}\, \id\bigs) * \mu^X.
\label{eq: Y(n)repre}
\end{equation}

Note that this representation of $Y^k$ is the same for \emph{any} starting process $X$, 
for each $k \in \N$.  For example, let $X_t = \mu t + \sigma W_t$ for all $t \geq 0$, where 
$W$ is Brownian motion with $W_0 = 0$. Here $\mu \in\R$ denotes the drift rate and  
$\sigma \in \R$ the volatility. Then \eqref{eq: Y(n)repre} yields
\begin{align*}
Y^k_t =\alpha ^{k} \sigma W_t + \left(\alpha^k \mu  + \frac{1}{2}\alpha ^{k+1}\frac{\alpha ^{k}-1}{\alpha
-1} \sigma^2 \right)t,  \qquad t \geq 0,
\end{align*}
for all $k \in \N$. 
Classical calculus would yield the same result, of course. For each $k \in \N$, one would repeatedly compute
\begin{equation*}
Y^k =\exp\Big(-\alpha Y^{k-1}_-\Big)\cdot \exp\Big(\alpha Y^{k-1}\Big).
\end{equation*}
This is not too complicated but can easily become quite cumbersome, even in the case of drifted Brownian motion. \qed
\end{example}

\begin{example}[It\^o--Wentzell formula] \label{Ex:IWF}
The semimartingale representation proposed in this paper naturally leads to a parsimonious generalization of the  It\^o--Wentzell formula; 
see Jeanblanc, Yor, and Chesney~\cite[Theorem~1.5.3.2]{jeanblanc.al.09} and also Bank and Baum~\cite[Proposition~1.3]{bank.baum.04}. 
To this end, consider an $\R^n$--valued semimartingale $V$ and a predictable function $\psi$ such that $\psi({x},\cdot)\in \I(V)$ 
for each ${x}\in\R^d$. Define next a family of semimartingales $(F({x}))_{ x \in \R^d}$ by setting
$$ F({x}) = \psi\left({x},\id \right)\circ V.$$

One can now randomize the family $F$ by allowing ${x}$ to switch values stochastically in line with the $\R^d$--valued semimartingale $X$. Assuming $F$ is sufficiently smooth, the randomized process $F(X)$ defined by 
$$ F(X)_t = F_t({x})|_{x = X_t}, \qquad t \geq 0$$
will again be a semimartingale. The observation 
$$\Delta F(X) = (F(X) - F_-(X)) +  (F_-(X) -F_-(X_-))$$ 
then yields, under suitable technical conditions, that 
\begin{equation}\label{eq:181211.1} 
F(X)=\xi \circ \left( X,V\right)
\end{equation}
with 
$$\xi (x,v)=\psi\left(X_{-}+x,v\right) + \left( \psi (\theta+x,\id \right) -\psi \left(\theta,\id)\right)\circ V_{-}|_{\theta=X_-}, 
\qquad x \in \R^d,\, v \in \R^n.$$

We leave the technical details to future work. For the moment, we only note that for  $\R$--valued continuous processes $X$ and $V$ and for $\psi({x},{v})= f({x}) v$, where $x,v \in\R$ and $f: \R \rightarrow \R$ is twice continuously differentiable, one formally obtains
\begin{alignat*}{3}
D_{1}\xi (0, 0) &=  f'(X)\, \id \circ V = f'(X)\cdot V;      \qquad & D_{2}\xi (0,0) &=f(X);      \qquad & &\\
D^2_{1,1}\xi (0, 0) &= f''(X)\,  \id \circ V =f''(X)\cdot V; \qquad & D^2_{1,2}\xi(0,0) &=f'(X);  \qquad & D^2_{2,2} \xi(0, 0)&= 0.
\end{alignat*}
If one can now show that $f'(X)\cdot V=F'(X)$ and $f''(X)\cdot V=F''(X)$, then 
 \eqref{eq:181211.1} yields the statement of Jeanblanc et al.~\cite[Theorem~1.5.3.2]{jeanblanc.al.09}. 
\qed
\end{example}

As the examples illustrate, the stochastic calculus introduced above is powerful and simple. Stochastic integration, It\^o's formula, and the composition rule of Theorem~\ref{T:composition} allow for a wide range of applications. Within the confines of their assumptions they show that it is enough to study jump transformations; i.e., to represent $Y$ in terms of $X$ it suffices to trace how the jump $\Delta X_t$ is transformed into the jump $\Delta Y_t$ at time $t \geq 0$.

\subsection{Counterexamples}
This subsection illustrates the tightness of the results in Section~\ref{S:3} by providing several counterexamples. 
\begin{example}[$\xi \in \I(X)$, but $\xi'(0) \notin L(X)$] \label{E:190916}
Here, we construct a process $X \in \V^{\d}_\sigma$ and a predictable function $\xi \in \I(X)$, twice continuously differentiable at zero, such that $\xi'(0) \notin L(X)$.  This illustrates the role of the predictable set $\mathcal H_X$ in Definition~\ref{D:170714}.

	Let $U \in \V^{\d}_\sigma$  denote a piecewise constant martingale that jumps at times $2 - \sfrac{1}{n}$ by $\pm 1/n^2$. Let $(\Theta_n)_{n \in \N}$ denote an independent sequence of independent $\{0,1\}$--valued random variables with $\P[\Theta_n=1] = \sfrac{1}{n^4}$. Let $(\Psi_n)_{n \in \N}$ denote a sequence, independent of $U$ and $(\Theta_n)_{n \in \N}$, of independent standard normally distributed random variables. Let now $V \in \V^{\d}_\sigma$  denote a piecewise constant martingale that jumps at times $2 - \sfrac{1}{n}$ by $\Psi_n$ if $\Theta_n = 1$ and does not jump if  $\Theta_n = 0$.

Next, set $X = U + V \in \V^{\d}_\sigma$ and assume the filtration is the natural filtration of $X$. An application of Borel--Cantelli then yields that $V$ only has finitely many jumps, hence $\Delta X = \Delta U$ except finitely many times. Consider next the deterministic predictable function $\xi$ given by 
\[
	\xi_t = \indicator{t<2}\left(\id \indicator{\abs{\id} \geq (2 -t)^2}  +  \frac{1}{(2-t)^2} \id \indicator{\abs{\id} < (2 -t)^2}  \right), \qquad  t \geq 0.
\]
Note that 
$$\xi'_t(0)= \indicator{t<2}  (2-t)^{-2},\qquad t \geq 0;$$ 
hence $\xi'_{2 - 1/n}(0) = n^2$ for all $n \in \N$ and $|\xi'(0) \id|^2 * \mu^X_2 = \infty$. 
Then thanks to Propositions~\ref{P:190330}, \ref{P:190411}, \ref{P:190911}, and \ref{P:1_J}\ref{P:1_J:ii}, $\xi'(0) \notin L(X)$ but $\xi \in L_\sigma(\mu^X) =  \I(X)$.

Moreover, if $Y = \xi \circ X$ also satisfies $Y = \eta \circ X$ for some $\eta \in \I(X)$ and $\eta'(0)$ exists, then $\eta_{2-1/n} = \xi_{2-1/n}$ for all $n \in \N$, hence  also $\eta'(0) \notin L(X)$.
\qed
\end{example}

\begin{example}[$\xi \in \I(X)$ and $\psi \in \I(\xi \circ X)$, but $\psi(\xi) \notin \I(X)$; also: $\xi \in \I(X)$, $\zeta \in L(\xi \circ X)$, but $\zeta \xi \notin \I(X)$]\label{Ex:1}
	Consider a continuous semimartingale $X$ given by 
	\begin{align*}
		X_t = W_ t- \int_0^t s^{-\sfrac{2}{3}} \d s = W_t - 3 t^{\frac{1}{3}}, \qquad  t\geq 0,
	\end{align*}
	where $W$ denotes a standard Brownian motion. Define the predictable, indeed, deterministic functions
	\begin{alignat*}{3}
		\xi_t &= \id +\id^2\, t^{-\sfrac{2}{3}}  \indicator{t>0};  
		&\qquad \psi_t &= \id\,  t^{-\sfrac{1}{3}}  \indicator{t>0}, &\qquad\qquad t &\geq 0.
	\end{alignat*}	
  Thus, $\xi \in \I^1(X)$; in particular, $Y = \xi \circ X$ satisfies $Y = W$. Hence, also $\psi \in \I^1(Y)$ and
$	\psi \circ Y  = \int_0^\cdot s^{-\sfrac{1}{3}} \d W_s$.

Now let $\eta = \psi(\xi)$. Thanks to \eqref{eq:200521.2} we have $\eta'_t (0) = t^{-\sfrac{1}{3}} $ for all $t > 0$;
hence, $\eta \notin \I(X)$ despite $\xi \in  \I(X)$ and $\psi \in  \I(\xi \circ X)$.
In this example, $\psi'(0)$ is deterministic, but \eqref{eq:200521.1}   does not hold. Hence, there is no contradiction to Theorem~\ref{T:composition}. \qed
\end{example}

\begin{example}[Alternative construction: $\xi \in \I(X)$ and $\psi \in \I(\xi \circ X)$, but $\psi(\xi) \notin \I(X)$; also: $\xi \in \I(X)$, $\zeta \in L(\xi \circ X)$, but $\zeta \xi \notin \I(X)$]\label{Ex:1''}
Let $(\tau_k)_{k \in \N}$ be a sequence of independent random variables with $\tau_k$ uniformly distributed on $(\sfrac{1}{(k+ 1)}, \sfrac{1}{k})$. Let $(J_k)_{k \in \N}$ be an independent sequence of independent $\{2,4\}$--valued random variables with $\P[J_k = 2] = \sfrac{1}{2} = \P[J_k = 4]$.   Set now
\[
	X_t = t^3 + \sum_{k = 1}^{\infty} \tau_k^{J_k} \indicator{\lc \tau_k, \infty\lc}(t), \qquad t \geq 0,
\]
and assume that the filtration be the right-continuous modification of the one generated by the finite-variation process $X$.

Consider now deterministic $\xi$ and $\psi$ given by $\xi_t = \indicator{\id \leq t^4} \id$ and $\psi_t = \sfrac{\id}{t^2} \indicator{t>0} $ for all $t \geq 0$.  Then $\xi \in \I(X)$ with 
\[
	Y_t = \xi \circ X_t = t^3 + \sum_{k = 1}^{\infty} \tau_k^{4} \indicator{\{J_k = 4\}} \indicator{\lc \tau_k, \infty\lc}(t), \qquad t \geq 0,
\]
and $\psi \in \I(Y)$ with 
\[
	\psi \circ Y_t = 3 t + \sum_{k = 1}^{\infty} \tau_k^{2} \indicator{\{J_k = 4\}} \indicator{\lc \tau_k, \infty\lc}(t), \qquad t \geq 0.
\]
However, it is clear that $\psi(\xi) \notin \I(X)$.  An even stronger statement holds, namely that there exists no $\eta \in \I(X)$ such that $\psi \circ Y  = \eta \circ X$. 
 \qed
\end{example}

\begin{example}[$\xi^{-1} \notin \I(\xi \circ X)$] \label{Ex:1a}
	Assume that $X = W$ is standard Brownian motion.
	Let $\xi$ denote some deterministic predictable function that satisfies $\xi_t(x) = tx + \sfrac{x^2}{2}$ for all $t>0$ and $x$ in a neighbourhood of zero (which may depend on $t$) and allows for an inverse. Then $\xi'_t(0) = t$  and $\xi''_t(0) = 1$ for all $t > 0$. Hence $\xi \in \I^1(X)$ and we can define $Y = \xi \circ X$, satisfying 
	\begin{align*}
		Y_t = \int_0^t s \d W_s + t, \qquad t > 0.
	\end{align*}
 Moreover, with $\psi = \xi^{-1}$ we have  $\psi(\xi) \in \I^1(X)$.  
	Observe, however, that $\psi'(0) \notin L(Y)$ and  $\psi''(0) \notin L([Y,Y]^c)$ since 
$\psi'_t(0) = \sfrac{1}{t}$ and  
$$\psi''_t(0) = \frac{- \xi''_t(0)}{( \xi'_t(0))^3} = - t^{-3}, \qquad t \geq 0.$$
Thus $\psi \notin \I(Y)$ but there is no contradiction to Remark~\ref{R:200521.2} as \eqref{eq:200521.1} is not met. Because $\psi'(0)=\sfrac{1}{\xi'(0)}$ is not locally bounded, this example does not contradict Corollary~\ref{C:170726} either.  Note that there exists no $\eta \in \I(Y)$ such that $X = \eta \circ Y$.
\qed
\end{example}

\begin{example}[Alternative construction: $\xi^{-1} \notin \I(\xi \circ X)$; additionally $\xi \circ X = X$]\label{Ex:1a''}
Let $(\tau_k)_{k \in \N}$ be a sequence of independent random variables with $\tau_k$ uniformly distributed on $(\sfrac{1}{(k+ 1)}, \sfrac{1}{k})$. 
Let $(U_k)_{k \in \N}$ be an independent sequence of independent and identically distributed $\{-1,1\}$--valued random variables with $\P[U_1 = 1] = \sfrac{1}{2}$.
 Set now
\[
	X =  \sum_{k = 1}^{\infty} U_k \tau_k \indicator{\lc \tau_k, \infty\lc} \in \V^{\d}_\sigma,
\]
and assume that the filtration be the right-continuous modification of the one generated by  $X$.

Let $\xi$ denote some deterministic predictable function that allows for an inverse and satisfies, for all  $t>0$,   $\xi_t(x) = x$ for all $x$ with $\abs{x} \geq \sfrac{1}{(t+1)}$ and $\xi_t(x) = tx$ for all  $x$ in a neighbourhood of zero (which may depend on $t$). Then $\xi \in \I(X)$ and $X = \xi \circ X$.  However,
since $\indicator{\H_X}=1$ and $D \xi^{-1}(0) \notin L(X)$, we have  $\xi^{-1} \notin \I(X)$, concluding the example. \qed
\end{example}

\begin{example}[$\xi, \psi(\xi) \in\I(X)$ and $\psi \in \I(\xi \circ  X)$ but $\psi(\xi) \circ X \neq \psi \circ (\xi \circ X)$] \label{Ex:3}
Let $X$ be a continuous semimartingale not equal to the zero process.  Consider $\xi = \id^3 \in \Uni$ and $\psi = \xi^{-1}$. Note that $\psi(\xi) =  \id \in \Uni$ and that $\xi \circ X =0$, hence $\psi \in \I(\xi \circ X)$. 
However, we have
\[
	\psi (\xi) \circ X = X \neq 0 = \psi \circ (\xi \circ X).
\]
Theorem~\ref{T:composition} is not contradicted because $\psi$ is not differentiable at zero, $(\P \times A^X)$--a.e.
\qed
\end{example}

\setlength{\abovedisplayskip}{5pt}
\setlength{\belowdisplayskip}{5pt}

\section{Predictable characteristics}\label{S:5}
Up to this point we have relied on a `pathwise' perspective  in the sense that the representation of the process $Y$ by means of $\xi\circ X$ depends on the probability measure only through the null sets; see also Remark~\ref{R:181106}.  Now we will demonstrate the ability to convert an $X$--representation into predictable characteristics. In this section, we shall use generalized conditional expectation; see  Shiryaev~\cite[pp.~475--476]{shiryaev.96} and Jacod and Shiryaev~\cite[I.1.1]{js.03}.

\subsection{Truncation functions}\label{sect: technicalitiesII}
In \cite[II.2.3--4]{js.03},  a bounded function $h: \Cx^d \rightarrow \Cx^d$ is called a truncation function for $X$ if $h(x) = x$ in a neighbourhood of zero. For such $h$ the process 
$X[h] = X - \left(\id - h\right)* \mu^X$
only has bounded jumps and is therefore special. Below, it will be  useful to not only control the jumps of $X$, but also those of a stochastic integral with respect to $\widehat{X[h]}=\hat\id(X[h])$. This leads to the following generalization of the classical  truncation function where the boundedness and integrability requirements are relaxed. Moreover, $h$ is no longer restricted to be a time-constant deterministic predictable function.

\begin{definition}[Truncation function for $X$ and its compatibility with $\xi\in\I(X)$]  \label{D:190628}
We call a predictable function $h:\widebar{\Omega}^d\rightarrow \Cinf^d$ a truncation function for $X$	if 
	\begin{align*}
		\id - h \in L_\sigma(\mu^X) \qquad \text{and if }  \qquad X[h] = X - \left(\id - h\right)\star \mu^X \text{ is special.}
	\end{align*}
Moreover, for $\xi\in\I(X)$, we say that a truncation function $h$ for $X$ is $\xi$--compatible if 
$\indicator{\H_X} \hat D\xi(0)\in L(\widehat{X[h]})$ and $\indicator{\H_X} \hat D\xi(0) \cdot \widehat{X[h]}$ is special.
\qed
\end{definition}

\begin{remark}[Observations on truncation functions]\label{R:190423.2}
If $h$ is a truncation function for $X$ and $\xi\in \I(X)$, implying $\indicator{\H_X} \hat D\xi(0)\in L(\hat X)$, it does not follow that $\indicator{\H_X} \hat D\xi(0)\in L(\widehat{X[h]})$. Indeed, there exists an $\R$--valued quasi-left-continuous  process $V\in\V^{\d}_\sigma$ with $V_0 = 0$ whose jumps are bounded and a predictable process $\zeta \in L(V-B^V)$ such that $\zeta\notin L(B^V)$ (see \cite[Example~3.11]{cr0}). Let  now 
$$X=V-B^V.$$ 
Then $h=0$ is a truncation function for $X$ with $X[0]=-B^V$ and $\H_X^c$ is empty. The predictable function $\xi=\zeta\id$  satisfies $\zeta = \indicator{\H_X} {\xi}'(0)\in L(X)$ but $\indicator{\H_X}{\xi}'(0) \notin L(X[0])$.

Consider now the process $Y = \zeta \cdot X$.  We claim that $h = 0$ is not a truncation function for $Y$. Hence, this provides an example of a process $Y$ such that $Y[0]$ does \emph{not} exist. Assume it did. Then $\id \in L_\sigma(\mu^Y)$, yielding $\zeta\id \in L_\sigma(\mu^X) = L_\sigma(\mu^V)$. In view of Proposition~\ref{P:190411} and the fact that $\zeta \notin L(V)$ this yields a contradiction.

	As a final observation for the moment, note that  the process $\indicator{\H_X} \hat D \xi(0)\cdot \widehat{X[h]}$ may not be special even if it is known that $h$ is a truncation function for $X$, 
	$\xi\in{\I}(X)$, and  $\indicator{\H_X}  \hat D\xi(0)$ integrates $\widehat{X[h]}$.  For example, let  $X = \indicator{\lc U, \infty\lc}$, where $U$ is uniformly distributed, let  $\filt F$ denote the smallest right-continuous filtration that makes $X$ adapted,  let $\xi_t = \sfrac{\id}{t}$ for all $t > 0$, and let $h = \id$. Then $h$ is a truncation function for $X$ with $X = X[h]$, $\xi \in \I(X)$, and $\H_X^c$ empty, but $\indicator{\H_X}  \xi'(0) \cdot X[h] =  \sfrac{1}{U} \indicator{\lc U, \infty\lc}$ is not special.
\qed
\end{remark}

\begin{lemma}[Compatible truncation]\label{L:compatible h}
There exists a $\xi$--compatible truncation function $h$ for $X$, for every $\xi \in \I(X)$. Moreover, $h$ can be chosen 
such that $h\in\I(X)$ with 
\begin{equation}\label{eq:190702.1}
X[h] = X_0 + h\circ X.
\end{equation}
Furthermore, if $\hat D\xi(0)$ is locally bounded (in particular, if $\xi\in\Uni\cap \I(X)$), then any truncation function for $X$ is $\xi$--compatible.
\end{lemma}
\begin{proof}
 By assumption, $\varsigma = \indicator{\H_X} \hat D \xi(0)$ is in $L(\hat{X})$.  We claim that 
\begin{align}\label{eq:200612.1}
	h =  \id \indicator{\abs{\id} \leq 1 \text{ and } \left\lvert \varsigma\, \hat\id \right\rvert \leq 1}
\end{align}
has the desired properties. Indeed, $\id-h\in L(\mu^X)$ because both $X$ and $\varsigma \cdot \hat{X}$ have finitely many jumps larger than one in absolute value on any compact time interval. This also yields $\varsigma\in L\big(\hat{X} - \widehat{X[h]}\big)$ and consequently 
$\varsigma \in L(\widehat{X[h]})$. The jumps of $X[h]$ and $\varsigma \cdot \widehat{X[h]}$ are bounded by 1 in absolute value; therefore both processes are also special. 

Observe that $h(\omega,t,x)=x$ on a $(\omega,t)$--dependent neighbourhood of zero, $(\P \times A^X)$--a.e. This yields that $h$ is analytic at 0, $Dh(0)$ is an identity matrix, and $D^2 h(0)=0$. The representation formula \eqref{eq:190703.1} now gives
$$X_0+h\circ X = X  + (h-\id)\star \mu^X=X[h].$$

The final claim follows by localization.
\end{proof}

\begin{remark}[Truncation at zero]
The previous lemma shows that sufficiently many truncation functions can be applied via the natural formula \eqref{eq:190702.1}. We elect \emph{not} to make \eqref{eq:190702.1}  the \emph{only} way to truncate because \eqref{eq:190702.1} does not hold for $h=0$. Truncation at zero is convenient when $X$ has jumps of finite variation; more generally, it can be applied whenever $\id\, \in L_\sigma(\mu^X)$. \qed
\end{remark}

The next proposition recognizes that the \'Emery formula~\eqref{eq:190610.2} represents a whole spectrum of equivalent expressions where the jumps of $X$ can be dialled down in the first term of~\eqref{eq:190610.2} as long as they are equivalently modified in the last term of~\eqref{eq:190610.2}. In most applications, it is possible to choose as truncation one of the polar cases $h=0$ or $h=\id$; less frequently one may have to opt for an intermediate truncation such as $h=\id \indicator{\abs{\id}\leq 1}$; in full generality it may be necessary to use the compatible truncation \eqref{eq:200612.1}.   
\begin{proposition}[\'Emery formula involving truncation]  \label{P:truncation}
Fix $\xi \in \I^n(X)$ and let $g$ be a truncation function for $\xi\circ X$. Moreover, let $h$ be a $\xi$--compatible truncation function for $X$. Then  the following terms are well defined and we have
\begin{align}
(\xi\circ X)[g] &=  \indicator{\H_X} \hat D \xi(0) \cdot \widehat{X[h]} 
		+ \frac{1}{2} \hat{D}^2 {\xi}(0) \cdot \bigs[\hat X,\hat X\bigs]^c
		+ \left(g(\xi) -  \indicator{\H_X} \hat D \xi(0) \hat{h} \right) \star \mu^{X}   \label{eq:190627.3a}\\
		&= \indicator{\H_X} \check{D}\xi(0) \cdot {\widecheck{X[h]}}
	+  \frac{1}{2} \check{D}^2 {\xi}(0) \cdot \bigs[\check X,\check X\bigs]^c
	+ \bigs(g(\xi)  - \indicator{\H_X} \check{D} \xi(0) \check{h} \bigs) \star \mu^{X}. \label{eq:190627.3b}
\end{align}
If $\indicator{\H_X}  \xi$ is analytic at $0$, $(\P \times A^X)$--a.e., the following terms are well defined and we have
\begin{align}
(\xi\circ X)[g] &=  \indicator{\H_X}  D \xi(0) \cdot {X[h]} 
		+ \frac{1}{2} D^2 {\xi}(0) \cdot [X,X]^c
		+ \left(g(\xi) -  \indicator{\H_X}  D \xi(0) h \right) \star \mu^{X}.  \label{eq:190627.3c}
\end{align}
If $g$ satisfies $g(w)=w$ on an $(\omega,t)$--dependent neighbourhood of 0, $(\P \times A^X)$--a.e., then we also have  
$g(\xi) \in \I(X)$ and
\[\pushQED{\qed}
(\xi\circ X)[g] = g(\xi) \circ X.
\qedhere\popQED \]
\end{proposition}
\begin{proof}
Thanks to $g$ being a truncation function for $\xi\circ X$ and  $h$ being $\xi$--compatible (in conjunction with Proposition~\ref{P:190411}) we have $(g(\xi) -  \indicator{\H_X} \hat D \xi(0) \hat{h}  \in L_\sigma(\mu^X)$. It is now simple to establish \eqref{eq:190627.3a}. Next, \eqref{eq:190627.3b} and \eqref{eq:190627.3c} follow as in Proposition~\ref{P:190114.1}.

The additional hypothesis on $g$ yields $\hat D\hat g (0) = I_{2n}$  and $\hat D^2  g(0)=0$, hence $g\in\I(\xi\circ X)$. The last statement follows from the definition of $\I(X)$, the $\circ$--notation, and from Corollary~\ref{C:200522}.
\end{proof}

\subsection{Characteristics under the measure \texorpdfstring{$\P$}{P}}
Proposition~\ref{P:truncation} yields the next observation, which is the key step towards computing predictable characteristics of represented semimartingales.
\begin{proposition}[Drift of a truncated represented semimartingale]\label{P:170822.1}
Fix $\xi\in{\I}(X)$ and let $g$ be a truncation function for $\xi\circ X$. Moreover, let $h$ be a $\xi$--compatible truncation function for $X$. Then the following terms are well defined and
the predictable compensator of $(\xi \circ X)[g]$ under $\P$ is given by 
\begin{align}
B^{(\xi\circ X)[g]} &= \indicator{\H_X} \hat D \xi(0) \cdot B^{\widehat{X[h]}}
		+ \frac{1}{2} \hat{D}^2 {\xi}(0) \cdot \bigs[\hat X,\hat X\bigs]^c
		+ \left(g(\xi) -  \indicator{\H_X} \hat D \xi (0) \hat{h} \right) \star \nu^{X}   \label{eq:190317.1}\\
		&= \indicator{\H_X} \check{D}\xi(0) \cdot B^{\widecheck{X[h]}}
	+  \frac{1}{2} \check{D}^2 {\xi}(0) \cdot \bigs[\check X,\check X\bigs]^c
	+ \bigs(g(\xi)  - \indicator{\H_X} \check{D} \xi(0) \check{h} \bigs) \star \nu^{X}. \label{eq:190317.2}
\end{align}
If $\indicator{\H_X} \xi$ is analytic at $0$, $(\P \times A^X)$--a.e., the following terms are well defined and we have
\begin{align}
B^{(\xi\circ X)[g]} &=  \indicator{\H_X}  D \xi(0) \cdot B^{X[h]} 
		+ \frac{1}{2} D^2 {\xi}(0) \cdot [X,X]^c
		+ \left(g(\xi) -  \indicator{\H_X}  D \xi (0) h \right) \star \nu^{X}.  \label{eq:190704.4}
\end{align}
\end{proposition}
\begin{proof}
In \eqref{eq:190627.3a}, the last term is special since all the other terms are special. Hence \eqref{eq:190317.1} follows 
from Shiryaev and Cherny~\cite[Lemma~4.2]{shiryaev.cherny.02}.
Equations~\eqref{eq:190317.2} and \eqref{eq:190704.4} follow as in Proposition~\ref{P:190114.1}.
\end{proof}

\begin{remark}[Discrete-time and continuous-time components of a drift]
Recall the unique decomposition in Proposition~\ref{P:190729}. Consider now a predictable function $\xi\in\I(X)$. Proposition~\ref{P:1_J}\ref{P:1_J:iii} asserts $(\xi\circ X)^\q=\xi\circ X^\q$ and $(\xi\circ X)^\ddp=\xi\circ X^\ddp$. 
Next, suppose $\xi\circ X$ is special. By Propositions~\ref{P:190520} and \ref{P:1}\ref{P:1:i}, the drift at predictable jump times then takes a particularly simple form, namely,
\begin{equation}\label{eq:210612}
 B^{\xi\circ X^\ddp} = \sum_{\tau\in\mathcal{T}_X}  \E_{\tau-}[\xi_\tau(\Delta X_\tau)]\indicator{\lc\tau,\infty \lc}.
\end{equation} 
Observe that this formula is simpler than Proposition~\ref{P:170822.1} applied to $X^\ddp$ in place of $X$. 
Therefore, in practice, Proposition~\ref{P:170822.1} is used with $X=X^\q$ to obtain $B^{\xi\circ X^\q}$. One then has 
\[
	B^{\xi\circ X} =  B^{\xi\circ X^\q} + B^{\xi\circ X^\ddp}. 
\]

Finally, recall that $X^\q$ is quasi-left-continuous, hence $B^{\xi\circ X^\q}$ is continuous, yielding 
$$\Delta B^{\xi\circ X} = \Delta B^{\xi\circ X^\ddp}.$$ 
The literature employs the following weakening of \eqref{eq:210612}, typically with $\xi=\id$,
$$ \Delta B^{\xi\circ X} =  \int_{\Cx^d} \xi(x) \nu^X(\{\cdot\}, \d x);$$
see, for example, \cite[II.2.14]{js.03}.
\qed
\end{remark}

\begin{corollary}[Characteristics of a represented semimartingale]\label{C:170822.1}
Let $Y=Y_0+\xi\circ X$ for some $\xi \in \I^n(X)$. Then the semimartingale characteristics of 
$Y$ with respect to the truncation function $g$ for $Y$ are given  by
\begin{align}
		B^{Y[g]} &= B^{(\xi\circ X)[g]}; \nonumber\\
\bigs[\hat Y^{(k)}, \hat Y^{(l)}\bigs]^c &{}=\left(\hat D {\hat\xi}^{(k)}(0)^\top \hat D {\hat\xi}^{(l)}(0)\right) \cdot \bigs[\hat X,\hat X\bigs]^c	,\quad k,l \in \{1, \cdots, 2n\}; \label{eq:190627} \\
\bigs[\check Y^{(k)}, \check Y^{(l)}\bigs]^c &{}= \left(\check D {\check\xi}^{(k)}(0)^\top \check D {\check\xi}^{(l)}(0)\right) 
\cdot \bigs[\check X,\check X\bigs]^c	,\quad k,l \in \{1, \cdots, 2n\}; \label{eq:200602}
\end{align}
\vspace*{-0.4cm}
\begin{gather}
\label{eq:170813.3}
\begin{aligned}
	&\text{$\nu^Y$ is the push-forward measure of $\nu^X$ under $\xi$, that is, $\psi * \nu^Y = \psi(\xi) * \nu^X$}\\
	&\text{for all }\text{non-negative bounded predictable functions $\psi$ with $\psi(0)=0$.}
\end{aligned}
\end{gather}
\end{corollary}

\begin{proof}
Definition~\ref{D:181101} yields \eqref{eq:190627} and \eqref{eq:200602} then follows from \eqref{eq:190704.2} and \eqref{eq:190115.3} in view of the identity 
$$\bigs(\check D \xi^{(k)}(0) \cdot\check X\bigs)^*=\bigs(\check D \xi^{(k)}(0)\bigs)^* \cdot\bigs(\check X\bigs)^*
=\check D \xi^{*(k)}(0) \cdot\check X
	,\qquad k \in \{1, \cdots, n\},$$ 
where the superscript $*$ denotes again the complex conjugate. The statement in \eqref{eq:170813.3} follows from Proposition~\ref{P:170822.1} on observing that $\nu^Y(G) = B^{\indicator{G} \circ (\xi\circ X)}$, where $G = G_1 \times G_2$ with $G_1 \subset [0,\infty)$ predictable and $G_2$ a closed set in $\Cx^n$ not containing a neighbourhood of zero.
\end{proof}
When $\xi$ is of the form $\xi = f(X_- + \id) - f(X_-)$ for a twice continuously differentiable real-valued function $f$ and when $X$ is real, then Corollary~\ref{C:170822.1} reduces to the situation in Goll and Kallsen~\cite[Corollary~A.6]{Goll:Kallsen}. When $\xi=R\,\id$ for some $\R^{n\times d}$--valued matrix $R$ and $X$ is real-valued,  Corollary~\ref{C:170822.1} yields the statement of Eberlein, Papapantoleon, and Shiryaev~\cite[Proposition~2.4]{eberlein.al.09}.

\begin{example}[Generalized Yor formula continued]
We continue the discussion of Proposition~\ref{P:190701b}. Consider $\alpha, \beta \in \Cx$ and a $\Cx^2$--valued semimartingale $X$ satisfying the assumptions of Proposition~\ref{P:190701b}\ref{P:190701b.1} and additionally $X$ is stopped when $\Delta X^{(1)}= -1$ or $\Delta X^{(2)}= -1$.  We are interested in the drift of 
\[
	Y = \Log\left(\Exp \bigs(X^{(1)}\bigs)^\alpha\Exp \bigs(X^{(2)}\bigs)^\beta\right) = ((1+\id_1)^\alpha (1+\id_2)^\beta -1) \circ X,
\]
see \eqref{eq:190701.5}. Here $(1+\id_1)^\alpha (1+\id_2)^\beta -1$ belongs to $\Uni$ and is analytic at zero. Let  $g$ and $h$ denote truncation functions for $Y$ and $X$, respectively. Thanks to Lemma~\ref{L:compatible h} and \eqref{eq:190704.4} we now have
\begin{equation}\label{eq:200601}
\begin{split}
	B^{Y[g]} ={}& \alpha B^{X[h]^{(1)}} + \beta B^{X[h]^{(2)}}+ \frac{1}{2} \alpha (\alpha - 1) \bigs[X^{(1)}, X^{(1)}\bigs]^c
	+ \frac{1}{2}  \beta (\beta - 1) \bigs[X^{(2)}, X^{(2)}\bigs]^c \\
	&{}+ \alpha \beta \bigs[X^{(1)}, X^{(2)}\bigs]^c
		+ \left(g \left((1+\id_1)^\alpha (1+\id_2)^\beta -1\right) - [\alpha\,\, \beta] h\right) * \nu^X.
\end{split}
\end{equation}
Moreover, Corollary~\ref{C:170822.1} yields
\begin{align*}
	[Y,Y]^c = \alpha^2 \bigs[X^{(1)}, X^{(1)}\bigs]^c + \beta^2 \bigs[X^{(2)}, X^{(2)}\bigs]^c 
	          + 2 \alpha \beta \bigs[X^{(1)}, X^{(2)}\bigs]^c.
\end{align*}
For a direct derivation of \eqref{eq:200601} in the real-valued case when $\alpha = 1$ and $\beta = -1$, see for example Kallsen~\cite[Lemma~4.3]{Kallsen:2000}. \qed
\end{example}

\begin{example}[Example~\ref{E:190302} continued] \label{E:200609}
Consider for some $\alpha\in\Cx$ and $\Cx$--valued $X$ with $\Delta X\neq-1$, the representation
$Y = (\abs{1+\id}^\alpha-1)\circ X$. Assume for simplicity that $Y$ is special. The function $\abs{1+\id}^\alpha-1$ is in $\Uni$ but not analytic at 0. Example~\ref{E:190302}, Lemma~\ref{L:compatible h}, and \eqref{eq:190317.1} now yield
\[
B^Y = \alpha\cdot B^{\Re X[h]}+\frac{\alpha}{2}(\alpha-1)[\Re X,\Re X]^c+\frac{\alpha}{2}[\Im X,\Im X]^c
+ \left(\abs{1+\id}^\alpha-1-\alpha\Re h\right)*\nu^X 
\]
for any truncation function $h$ for $X$.\qed
\end{example}

\section{Concluding remarks}\label{S:6}
Let us review the benefits of the proposed `calculus of predictable variations.' Some of the advantages, such as universality of representations in $\Uni$ and the ease with which calculations can be performed in a very general class of complex-valued functions, have been showcased in the introduction and subsequently in the main body of the paper. Here we want to mention several other benefits that are of a more philosophical kind or whose detailed treatment is beyond the scope of this paper and will be pursued in other work. 

The literature has a number of fragmented and specialized results that fit into the framework of semimartingale representations. On their own, these results are hard to generalize and do not suggest fruitful unification, hence are also difficult to recall and disseminate. The new calculus overcomes this barrier by providing a compact, systematic way of recording existing (and new) results. Let us mention two classical examples to illustrate these advantages.
\begin{itemize}
\item Recall that a $\Cx$--valued continuous local martingale is called conformal if $[X,X]^c=0$. Hence by \eqref{eq:190703.1}, an analytic representation with respect to a continuous conformal local martingale is again a conformal local martingale. This not only covers a change of variables by means of an analytic function, as in Getoor and Sharpe~\cite[Proposition~5.4]{getoor.sharpe.72}, but includes \emph{arbitrary} representation analytic at the origin. For example, the stochastic logarithm of a natural exponential preserves continuous conformal local martingales as its representing function $\e^{\id}-1$ is analytic at 0.
\item Consider now the explicit characterization of the complex stochastic exponential due to Dol\'eans-Dade~\cite[Th\'eor\`eme~1]{doleans-dade.70}. This is captured by the representation \eqref{eq:190128.2},
$$\Exp(X) = \e^{\log(1+\id)\circ X},\qquad \text{provided $\Delta X\neq -1$}.$$
As $\log(1+\id)$ is in $\Uni$ and analytic at 0, the \'Emery formula~\eqref{eq:190703.1} yields
$$ \log(1+\id)\circ X = X-X_0 - \frac{1}{2}[X,X]^c + (\log(1+\id)-\id)*\mu^X, 	$$
hence the jump integral converges pathwise, $\P$--almost surely. After exponentiation this yields the aforementioned important formula
$$ \Exp(X) = \e^{X-X_0 - \frac{1}{2}[X,X]^c}\prod_{t\leq\cdot} \e^{-\Delta X_t}(1+\Delta X_t),$$
this time in full generality, because the jump to zero may be treated separately. 
\end{itemize}

Further advantages of the new calculus emerge when one is tasked with computing the drift of a represented process under some new probability measure $\Qu$ whose density $Z$ with respect to $\P$ is also represented, say by $\Log(Z)=\psi\circ X$. It now suffices to observe that by Girsanov's theorem the $\Qu$--drift of $X$ equals the  $\P$--drift of $X+[X,\Log(Z)]=\id(1+\psi)\circ X$. 
We refer the reader to \v{C}ern\'{y} and Ruf \cite{crIII} for a detailed treatment of measure changes by means of non-negative, represented, multiplicatively compensated semimartingales and once again to \cite{crI} for specific applications.

The suggested calculus has one other benefit for applied stochastic modelling. In an applied setting it is impractical to work with the raw characteristics
$$\left(B^{X[h]},[\hat X,\hat X]^c,\nu^{X}\right).$$ 
This issue can be addressed by decomposing the process $X$ uniquely into a `discrete-time' component $X^\ddp$ involving only jumps at predictable times and a `continuous-time' part $X^\q$, see Proposition~\ref{P:190729}. When it comes to computing drifts, the jumps at predictable times $\tau$ can be treated separately via the natural formula
$$ \Delta B_\tau^{\xi\circ X} = \E_{\tau_-}[\xi_\tau(\Delta X_\tau)].$$

The remaining quasi-left-continuous part $X^\q$ is usually an It\^o semimartingale in applications, i.e., the characteristics of $X^\q$ are assumed to be absolutely continuous with respect to time. One may then rephrase the drift computation for this component in terms of time rates, reverting to drift rates, quadratic variation rates (squared volatilities), and jump intensities (L\'evy measures). Thus, the calculus naturally accommodates the two most common ways of specifying the underlying stochastic process $X$ (discrete time vs.~an It\^o semimartingale) and even allows them to be combined in intricate ways, see \cite[Example~4.5]{crIII}. 

We shall close by mentioning possible directions for future research. As for extensions of the classes $\Uni$ and $\I(X)$, the most immediate generalization concerns the level of smoothness of the representing function at the origin. Lack of differentiability is associated with the need to consider local times in the It\^o--Meyer formula; see Karatzas and Shreve~\cite[Theorem~3.6.22]{KS1}. This suggests an appropriate modification of the \'Emery formula~\eqref{eq:190610.2}, for which the three key operations would have to be checked again. In Example~\ref{Ex:IWF}, we have broached the subject of the It\^o--Wentzell formula  that we believe merits further investigation. 

\def\MR#1{\href{http://www.ams.org/mathscinet-getitem?mr=#1}{MR#1}}
\def\ARXIV#1{\href{https://arxiv.org/abs/#1}{arXiv:#1}}
\def\DOI#1{\href{https://doi.org/#1}{doi:#1}}
\medskip

\noindent\textbf{Acknowledgements.}
We thank Jan Kallsen, Christoph K\"uhn, Johannes Muhle-Karbe, Pietro Siorpaes, two anonymous referees, and an associate editor for helpful comments and suggestions.

\end{document}